\documentclass[10pt]{article}
\usepackage{amsmath,amsthm,amsfonts,amssymb}
\usepackage{graphicx}
\usepackage{color}
\newtheorem{theorem}{Theorem}[section]
\newtheorem{lemma}[theorem]{Lemma}
\theoremstyle{definition}
\newtheorem{algorithm}[theorem]{Algorithm}
\newtheorem{corollary}[theorem]{Corollary}

\newtheorem{definition}[theorem]{Definition}

\newtheorem{problem}[theorem]{Problem}
\newtheorem{remark}[theorem]{Remark}
\newtheorem{proposition}[theorem]{Proposition}

\def\CC{{\cal C}}

\def\CP{{\cal P}}
\def\CE{{\cal E}}

\def\CL{{\cal L}}
\def\CT{{\cal T}}
\def\CF{{\cal F}}

\def\Lp{{\overline{p}}}
\def\Lq{{\overline{q}}}
\def\Le{{\overline{\varepsilon}}}
\def\Lr{{\overline{r}}}
\def\Ld{{\overline{\delta}}}

\def\MZ{{\mathbb{Z}}}
\def\MN{{\mathbb{N}}}

\DeclareMathOperator{\tower}{tower}
\DeclareMathOperator{\Reduce}{Reduce}

\newcommand{\gp}[1]{{\langle #1 \rangle}}
\newcommand{\rb}[1]{{\left( #1 \right)}}

\newenvironment{blue}{\color{blue}}{}

\title{Power Circuits, Exponential Algebra, and Time Complexity}

\date{\today}

\author{Alexei G. Myasnikov, Alexander Ushakov, and Dong Wook Won}

\begin{document}

\maketitle

\begin{abstract}

Motivated by algorithmic problems from combinatorial group theory
we study computational properties of integers equipped with binary
operations $+$, $-$, $z = x 2^y, z = x 2^{-y}$ (the former two are partial) and predicates
$<$ and $=$. Notice that in this case very  large numbers, which are obtained as  $n$ towers  of exponentiation in the base $2$
 can be realized as $n$ applications of the operation $x2^y$, so working with such  numbers given in the usual  binary
expansions requires super exponential space. We
define a new compressed representation for integers by power circuits (a particular type of straight-line programs) which is unique and easily computable,   and show that the operations above  can be performed in polynomial time if the numbers are presented by power circuits. We mention several applications of this technique to algorithmic problems, in particular, we prove that the quantifier-free theories  of various exponential algebras are decidable in polynomial time, as well as the word problems in some ``hard to crack'' one-relator groups.
\end{abstract}

\tableofcontents

\section{Introduction}

In this paper we study power circuits (arithmetic circuits with exponentiation), and show that a number of algorithmic problems in algebra, involving exponentiation, is solvable in polynomial time.

\subsection{Motivation}

Massive numerical computations play a very important part in modern science. In one way or another they are usually reduced to computing with integers. This unifies various computational techniques  over algebraic structures within the theory of constructive \cite{Malcev} or recursive models  \cite{Rabin}. From a more practical view-point these reductions allow one to utilize the fundamental mathematical fact that  the standard arithmetic manipulations with integers can be performed fairly quickly. In computations,  integers are usually presented in the binary form, i.e., by words in the alphabet $\{0,1\}$. Given two integers $a$ and $b$ in the binary form one can perform the basic arithmetic operations in time $O(N \log N \log\log N)$, where $N$ is the maximal binary length of $a$ and $b$ (see, for example, \cite{Bini_Pan:book:1990}).   In the modern mathematical jargon one can say that the structure $\mathbb{Z} = \langle \mathbb{Z}, +,-,\cdot, \, \leq \rangle$ (the  standard arithmetic) is computable in at most quadratic time with respect to the binary representation of integers, or it is polynomial time  computable (if we do not want to specify the degree of polynomials). This result holds for arbitrary $n$-ary representations of integers.

Notice, that the reductions mentioned above are not necessary computable
in polynomial time. In fact, there are recursive structures where
polynomial time computations  are impossible. Furthermore, there are
many natural algebraic structures that admit efficient computations,
though efficient algorithms are not easy to  come by  (we discuss some
examples below).  Usually, the core of the issue is to find a  specific
representation (data structure) of  given  algebraic objects which is
suitable for fast computations.

 For example, the standard representation of integer  polynomials from $\mathbb{Z}[X]$  as formal linear combinations of monomials in variables from $X$, may not be the most efficient way to compute with.   Sometimes, it is more computationally advantageous to represent polynomials by  {\em arithmetic circuits}. These circuits are finite directed labeled acyclic graphs $C$ of a special type.  Every node with non-zero in-degree in such a circuit $C$ is labeled either by  $+$ (addition) or by $-$ (subtraction), or by $\cdot$ (multiplication); nodes of zero in-degree (source nodes) are labeled either by constants from $\mathbb{Z}$ or by indeterminates $x_i \in X$. Going from the source nodes to a distinguished  sink node (of zero out-degree) one can write down a polynomial $p_C$ represented by the circuit $C$.  Observe, that given two circuits $C$ and $D$ it is easy to  construct a new circuit that represents the polynomial $p_C + p_D$ (or $p_C\cdot p_D$), so algebraic operations over circuits representations  are almost trivial. In \cite{Strassen:1969} Strassen used arithmetic circuits to design an efficient algorithm that performs matrix multiplication faster than $O(n^3)$. There are polynomial size circuits to compute determinants and permanents. We refer to a survey \cite{Strassen:1990} and a book \cite{Burgisser:1997} for results on algebraic circuits and complexity.

The idea to use circuits (graphs) to represent terms of some fixed functional language,
or the functions they represent, is rather general.  For example, {\em boolean circuits} are used to deal with  boolean formulas or functions. Here boolean formulas can be  viewed as terms in the language $\{\wedge, \vee, \neg, 0,1\}$ of boolean algebras. Construction of boolean circuits is similar to the arithmetic ones, where the arithmetic operations are replaced by the boolean operations and integers are replaced by the constants $0, 1$. Again, it is easy to define  boolean operations over boolean circuits,  but to check if two such circuits represent the same boolean function (or, equivalently, if a given circuit represents a satisfiable formula) is a much more difficult task (NP-hard). In 1949 Shannon suggested to use the size of a smallest circuit representing a given boolean function $f$ as a measure of complexity of $f$ \cite{Shannon:1949}. Eventually, this idea developed into a major area of modern complexity theory, but this is not the main subject of our paper.

Another powerful application of the ``circuit idea'' is due to Plandowski, who introduced compression of words in a given finite alphabet $X$ \cite{Plandowski:1994}. These compressed words can be realized by circuits over the free monoid $X^*$, where  the  arithmetic operations are replaced by the monoid multiplication, so every such circuit $C$  represents a word $w_C \in X^*$. The crucial point here is that the length of the word $w_C$ can grow exponentially with the size of the circuit $C$, so the standard word algorithms become time-consuming. For instance, the direct algorithm to solve the comparison problem  (if $w_C = w_D$ for given circuits $C, D$)  requires exponential time,  though there are  smart polynomial time (in the size of the circuits) algorithms that can do that  \cite{Plandowski:1994}.  In \cite{Lohrey:2004} Lohrey proved similar results for reduced words in a given free group (in the group language). This brought a whole new host of efficient algorithms in group theory \cite{Schleimer:2008}.

\subsection{Algorithmic problems for algebraic circuits}

In view of the examples above, we introduce here a general notion of an  algebraic circuit and related algorithmic  problems. Let $\mathcal{L}$ be a finite set of symbols of operations (a functional language).  An   {\em algebraic circuit} $C$ in $\mathcal{L}$ (or an $\mathcal{L}$-circuit) is a finite directed graph  whose nodes are either {\em input nodes} or {\em  gates}. The inputs nodes  have in-degree zero and  are each labeled by either variables  or constants from  $\mathcal{L}$; each gate is   labeled by an operation from  $\mathcal{L}$  whose arity equals to the in-degree of the gate; vertices of out-degree zero are called {\em output nodes}.  For a distinguished output vertex
in $C$  one can associate a term $t_C$ as was described above.
 Observe, that this notion of an algebraic circuit is more general than the usual one (see, for example,  \cite{Arora_Barak,Strassen:1990}), where  $\mathcal{L}$ is either the ring  or field theory language. On the other hand, algebraic circuits can be viewed also as {\em straight-line programs} in $\mathcal{L}$ (see \cite{AU:1970} by Aho and Ullman). In our approach to algebraic circuits we  follow  \cite{AU:1970}, even though in this case the orientation of edges is reversed, but this should not confuse the reader.

 There are several basic algorithmic problems associated with  $\mathcal{L}$-circuits over a fixed algebraic structure $A$ in  $\mathcal{L}$.  Denote by $Const(A)$ the set of elements of $A$ which are specified in $\mathcal{L}$ as constants.
  The {\em value problem} (VP) is to find the value of the term $t_C$ under an assignment of variables $\eta: X \to Const(A)$ for a  given $\mathcal{L}$-circuits $C$.
  The {\em value comparison problem} (VCP), mentioned above, is to decide if the terms $t_C$ and $t_D$ take the same value in $A$ under the assignment $\eta$ for given $\mathcal{L}$-circuits $C$ and $D$. For a functional language $\mathcal{L}$, decidability of VCP in $A$ implies decidability of the quantifier-free theory $Th_{qf}(A)$ of $A$.
  More generally, one may allow any assignments $\eta$ with   values in a fixed subset $S$ of $A$. In this case  decidability of VCP in $A$ relative to $S$ is equivalent to decidability of the quantifier-free theory of $A$  in a language $\mathcal{L}_S$ (obtained from $\mathcal{L}$ by adding  constants from $S$).

 If the language $\mathcal{L}$ contains predicates then decidability of $Th_{qf}(A)$ depends completely on decidability of the set of atomic  formulas  in $A$. Recall, that atomic formulas in $\mathcal{L}_S$ are of the form $P(t_1^\eta, \ldots, t_n^\eta)$, where $P$ is a predicate  in $\mathcal{L}$ (including equality) and $t_i^\eta$ is  evaluation of a term $t_i$   under an assignment $\eta$. Notice, that if $A$ is recursive then all the problems above are  decidable in $A$ in the language $\mathcal{L}_A$.

 From now on we deal  only with recursive structures, and our main concern is the time complexity of the decision problems. This brings an important new twist to decision problems. It might happen that the direct evaluation of $t_i$ under $\eta$ is time consuming, so we prefer to keep $t_i^\eta$ in the ``compressed form'' $t_{C_i}$ for some $\mathcal{L}$-circuit $C_i$ and proceed to  checking  whether or not the formula $P(t_{C_1}, \dots, t_{C_n})$ holds in $A$ without computing the values $t_i^\eta$. This is the essence of our approach to computational problems in this paper --  we operate with terms $t$ in their compressed form $C_t$ to speed up computations. Such approach makes the following {\em term-realization  problem} crucial: for a given term  $t(x_1, \ldots, x_n)$ in $\mathcal{L}$ construct in polynomial time an $\mathcal{L}$-circuit $C$  such that $t_C$ gives the function defined by $t$ in $A$. A related  {\em term-equivalence problem} asks   for given $\mathcal{L}$-circuits $C$ and $D$ if  the functions defined in $A$  by $t_C$ and $t_D$ are equal or not. Observe, that $t_C = t_D$ in $A$ if and only if the identity  $\forall X (t_C(X) = t_D(X)$  holds in $A$. So decidability of the term-equivalence problem in $A$ is equivalent to decidability  of the equational theory of $A$ (the set of all identities in $\mathcal{L}_S$ which hold in $A$).

\subsection{Exponential algebras}

In this paper we introduce and study algebraic circuits in exponential algebras. Typically, every such  algebra has a unary exponential function $y = E(x)$  as an operation, besides the standard ring operations of addition and multiplication. Some variations are possible here, so the language may contain additional operations (subtraction, division, multiplication by a power of $2$, etc.) or predicates (ordering, divisibility,  divisibility by a power of $2$, etc.). We refer to such language, in all its incarnations, as to {\em exponential algebra}  language and denote it by $\mathcal{L}_{exp}$. Exponential algebra is a very active part of modern algebra and model theory, it stems from two Tarski's problems. The first one, The High School Algebra Problem, is about axioms  of the equational theory of the high school arithmetic, i.e., the structure $N_{HS} = \langle \mathbb{N}_{>0}; +,\cdot, x^y,1\rangle$, where $\mathbb{N}_{>0}$ is the set of positive integers. Namely, it asks if every identity that holds on $N_{HS}$ logically follows from the classical ``high school axioms'' (introduced by Dedekind in  \cite{Dedekind}).
This problem was settled in the negative by Wilkie in
\cite{Wilkie:2000}, where he gave an explicit counterexample.
Moreover, it was shown that the equational  theory of $N_{HS}$    is not finitely axiomatizable, though decidable (Gurevich \cite{Gurevich:1990} and Macintyre \cite{Macintyre:1981}).   The time complexity of the problem  is unknown. Effective manipulations with terms over $N_{HS}$ are important in numerous applications, it suffices to mention such programs as Mathematica, Maple, etc.

The second Tarski's problem asks whether or not  the elementary theory of the field of reals $\mathbb{R}$ with the exponential function $y = e^x$ in the language is decidable. In the paper \cite{Macintyre_Wilkie:1996} Macintyre and Wilkie proved that the elementary theory of $(\mathbb{R}, e^x)$ is decidable provided the Schanuel's Conjecture holds. The time complexity of the  quantifier-free theory of $(\mathbb{R}, e^x)$ or the term-equivalence problem for algebraic circuits over $(\mathbb{R}, e^x)$  is  unknown (see \cite{Richardson:1969,Richardson:1983,Richardson:1992} for related problems).

\subsection{Our results}

Our main results here concern with  the time complexity of the quantifier-free theory of the typical exponential algebras over natural numbers.
We show that the quantifier-free theory is decidable in polynomial time in a structure ${\tilde N} = \langle \mathbb{N}_{>0}; +,  x\cdot 2^y,  \leq,  1\rangle$, a slight modification of the high-school arithmetic $N_{HS}$,
where, exponentiation and multiplication are replaced by $ x\cdot 2^y$ and the ordering predicate $\leq$ is included.  Of course, substituting 1 for $x$ one gets the exponential function $2^y$.  We show that the term-realization problem in $\tilde N$ is decidable  in polynomial time, as well as the quantifier-free theory $Th_{qf}({\tilde N})$.  This is precisely the case when  the direct evaluation of a term for a particular assignment of variables might result in a superexponentially long number, so  we avoid any direct evaluations of terms and work instead with  algebraic circuits.
The result holds if the partial function $x\cdot 2^{-y}$ is added to the language. In this event  for every quantifier-free sentence  one can decide in polynomial time  whether or not  it holds in $A$, or  is undefined.
Strangely, the  methods we exploit fail for the term-realization problem in the classical high-school arithmetic $N_{HS}$, the size of the resulting circuit may grow exponentially.

The Tarski's problem on decidability  of $(\mathbb{R},e^x)$ generated  very interesting research on exponential rings and fields (see, for example, \cite{Dries1,Macintyre:1991,Macintyre:1991(2),W,Macintyre_Wilkie:1996,Zilber:2004}). In \cite{Dries1,Macintyre:1991(2)} a free commutative ring with exponentiation $\mathbb{Z}[X]^E$  (with basis $X$) was constructed -- a free object in the variety of commutative unitary rings with an extra unitary operation for exponentiation $y = E(x)$. To perform various manipulations with exponential polynomials (elements of $\mathbb{Z}[X]^E$ ) it is convenient to use power circuits, i.e., algebraic circuits over an algebraic structure $\tilde Z = \langle \mathbb{Z}; +, -,  x\cdot 2^y,  \leq,  1\rangle$. The results described above  for $\tilde N$ hold also  in $\tilde Z$, so the term-realization problem and the quantifier-free theory of $\tilde Z$ are  decidable in polynomial time. Whether  these results hold with the  multiplication in the language is an open problem.

In fact, our technique gives decidability in polynomial time of the term-realization problem and the quantifier-free theory of the classical exponential structures $N_{HS}$ and $Z_{exp} = \langle \mathbb{Z}; +, -, x\cdot y, 2^y,  \leq,  1\rangle$ even with the multiplication in the language if one considers only terms in the standard form, i.e., if they are given as exponential polynomials (see \cite{Dries1,Macintyre:1991(2)}).

  All the results mentioned above also hold if the exponentiation in the base $2$ is replaced by an exponentiation in an arbitrary base $n \in \mathbb{N}, n \geq 2$. The argument for base $2$ goes through in the general case as well.

Another application of power circuits comes from the theory of automatic structures, that was introduced by Hodgson \cite{Hodgson:1976}, and Khoussainov and Nerode \cite{KN:1995} (we  refer to a recent survey \cite{Rubin:2008} for details). Automatic structures  form a nice subclass of recursive structures with decidable elementary theories.  Arithmetic with weak division $N_{weak} = \gp{N; S, +, \leq, |_2 }$,  where $x |_2 y$  if and only if $x$ is a power of 2 and $y$ is a
multiple of $x$ (``weak division''), is an important  example of an automatic structure.  It  has the following universal property (see Blumensath and  Gradel \cite{BG:2000}):  an arbitrary  structure $A$ has an automatic presentation if and only if it is interpretable (in model theory sense) in $N_{weak}$. This implies that first-order questions about automatic structures can be reformulated as first-order questions on $N_{weak}$. It is known that the first order theory of $N_{weak}$ is decidable, but its time complexity is non-elementary \cite{BG:2000}.  In view of the above, the complexity of the existential theory of $N_{weak}$ is an open problem of prime interest. Notice, that complexity of the problem depends on the representation of the inputs. It follows from our results on power circuits that the quantifier-free theory of $N_{weak}$ is decidable in polynomial time even when the numbers are presented in the compressed form by power circuits. We mention in passing that it would be interesting to see if the structure ${\tilde N} = \langle \mathbb{N}_{>0}; +,  x\cdot 2^y,  \leq,  1\rangle$ is automatic or not.

 We would like to mention one more application of power circuits, which triggered this research in the first place. In the subsequent paper we use power circuits to solve a well-known  open problem in geometric group theory. In 1969 Baumslag introduced (\cite{Baumslag:1969}) a one relator group
 $$G = \gp{a,b ~;~ (b^{-1}ab)^{-1}a(b^{-1}ab)=a^2},$$
 which later became one of the most interesting examples in geometric group theory.  It has been noticed by Gersten  that the Dehn function of $G$ cannot be bounded by any finite tower of exponents \cite{G1} (see complete proofs and upper bounds in the paper by  Platonov \cite{Platonov}). The Word Problem in $G$ is considered to be the hardest among all known one-relator groups. Recently,  Kapovich and Schupp showed in \cite{KS} that  the Word Problem  in $G$ is decidable in exponential time. Using power circuits we prove in \cite{MUW:2010} that the Word Problem in $G$ is polynomial time decidable.

All the results above are based on a new representation of integers,
which is much more ``compressed'' than the standard binary representation.
This ``power representation'' is interesting in its own sake. We represent
integers by  constant power circuits in the normal form. Such representation
is unique and easily computable: a number $n \in \mathbb{N}$ can be
presented by a normal power circuit $\CP_n$ of size at most $\log_2 n +2$,
and it takes time $O(log_2 n log_2log_2 n)$ to find $\CP_n$.  Furthermore,
we develop algorithms that allow one to perform the standard algebraic
manipulations (in the structure ${\tilde N}$ ) over integers given in
power representation in polynomial time.

\subsection{Outline}

In Section \ref{se:binary_sums} we introduce a  new way to represent integers as binary sums (forms) by allowing coefficients $-1$ in binary representations. In Section \ref{se:elementary_properties} we describe some elementary properties of these forms and  design an algorithm that compares numbers given in such binary forms in linear time (in the size of the forms). In Section
\ref{sec:compact-forms} we introduce ``compact'' binary sums  which give shortest possible representations of numbers, and show that these forms are unique.  It takes linear time (in the size of the standard binary representation) to compute the shortest binary form for a given integer $n$.

In Section \ref{se:power_circuits} we give a definition of a general algebraic circuit in the language
$\mathcal{L} = \{+,-,\cdot, x\cdot2^y\}$ and  define a special type of circuits,
called {\em power circuits}. Power circuits are main technical objects of the paper.
We show in due course that every algebraic circuit in $\CL$ is equivalent in
the structure $\tilde Z = \langle \mathbb{Z}, +,-,\cdot, x\cdot2^y\rangle$
to a power circuit, but power circuits are much easier to work with.
Besides, power circuits give a very compact presentation of natural
numbers, designed specifically for
efficient computations with exponential polynomials.

In Section \ref{se:circuit_types}  we define several important types of circuits:
standard, reduced and normal. The standard ones can be  easily obtained from general
power circuits through some obvious simplifications. The reduced power circuits
output numbers only, they require much stronger rigidity conditions (no redundant
or superfluous pairs of edges, distinct vertices output distinct numbers), which are
much harder to achieve. The normal power circuits are reduced and output numbers in
the compact binary forms. They give a unique compact  presentation of integers, which
is much more compressed (in the worst case) than the canonical binary representations.
This is the main construction of the paper, designed to  speed up  computations in
exponential algebra. We hope that the construction  is interesting in its own right.

In Section \ref{se:reduction} we describe a reduction process which for a given constant power circuit $\CP$ constructs an equivalent reduced power circuit $Reduce(\CP)$ in cubic time in the size of $\CP$. This is the main technical result of the paper.

In Section \ref{se:normal_form_computing}
we show how to compute the normal  power circuit representation of a given
integer $n$ (given in its binary representation) in time
$O(log_2 n ~log_2log_2 n)$.

In Section \ref{se:operations} we describe how to perform the standard arithmetic operations and exponentiations (in the  language $\mathcal{L}$)   over integers given in their power circuit representations. It turns out that the size of the resulting power circuits grows linearly, except for the ones produced by the multiplication (this is the main difficulty when dealing with power circuits).
Finally, we show how to compare (in cubic time) the  values of given constant
power circuits without producing the binary representations of
the actual numbers; and how to  find the  normal form of a given
constant power circuit  (in cubic time).

In Section \ref{se:exponential_algebra_circuits} we solve some problems mentioned earlier in the introduction. Fix a language $\CL = \{+,-,\ast, x\cdot 2^y, x\cdot2^{-y},\le, 0, 1\}$,  its sublanguage $\CL_0$, which is obtained from $\CL$ by removing the multiplication $\ast$; and structures $\MZ_\CL = \gp{\MZ;+,-,\ast, x\cdot 2^y, x\cdot2^{-y}, \le,1}$ and $\tilde Z =  \gp{\MZ;+,-, x\cdot 2^y, x\cdot2^{-y}, \le,1}$. We show that there exists an algorithm that for every algebraic $L$-circuit $C$ finds an equivalent standard power circuit $\CP$, or equivalently, there exists an algorithm which for every term $t$ in the language $\CL$ finds a power circuit $C_t$ which represents a term equivalent to the term $t$ in $\MZ_\CL$.  Moreover, if the term $t$ is in the language $\CL_0$ then the algorithm computes the circuit $C_t$ in linear time in the size of $t$.  For integers and closed terms in $\CL_0$ one can get much stronger results. Let $\CC_{norm}$ be the set of all constant normal power circuits (up to isomorphism). We show that if  $t(X)$ is a term in $\CL_0$ and $\eta:X \to \MZ$ an assignment of variables, then there exists an algorithm which determines if $t(\eta(X))$ is defined in $\MZ_\CL$ (or $\tilde Z$)  or not;
and if defined it then produces the normal
circuit $\CP_t $ that presents the number $t(\eta(X))$ in polynomial time. At the end of the section we prove that the  quantifier-free theory of the structure  $\tilde{Z}$  with all the constants from $\MZ$ in the language is decidable in polynomial time.

In Section \ref{se:difficulties}  we  demonstrate some inherent difficulties when dealing with  products of power circuits (the size of the resulting circuit grows exponentially).

Finally, in Section \ref{se:open_problems} we state some open
problems on complexity of algorithms in the classical exponential algebras.

\section{Binary sums}
\label{se:binary_sums}

In this section we introduce a  new way to represent integers as binary sums (forms) by allowing also coefficients $-1$ in binary representations. In Section \ref{se:elementary_properties} we describe some elementary properties of these forms and  design an algorithm that compares numbers given in such binary forms in linear time (in the size of the forms). In Section
\ref{sec:compact-forms} we introduce ``compact'' binary sums  which give shortest possible representations of numbers, and show that these forms are unique.  It takes linear time (in the size of the standard binary representation) to compute the shortest binary form for a given integer $n$.

\subsection{Elementary properties}
\label{se:elementary_properties}

A {\em binary term} $P(\overline{x}, \overline{y})$ is a term in the language $\{+, -, \cdot, 2^y\}$ (or $\{+, -, x\cdot2^y\}$) of the following type:
\begin{equation}\label{eq:bin_term}
  x_1 2^{y_1} + \ldots +  x_k 2^{y_k} \ \ (\mbox{which we also denote by } \sum_{i=1}^k x_i 2^{y_i}).
\end{equation}
Any assignment of variables $x_i = \varepsilon_i, y_i= q_i$
with  $\varepsilon_i \in \{-1,1\}$ and $q_i \in \mathbb{N}$ ($ i = 1, \ldots,k$) gives an algebraic expression,  called a {\em binary sum} (or {\em a binary form}),
\begin{equation}\label{eq:bin_decomp}
    \varepsilon_1 2^{q_1} + \ldots +  \varepsilon_k 2^{q_k},
\end{equation}
which we also denote by $\sum_{i=1}^k \varepsilon_i 2^{q_i}$ or $ P(\overline{\varepsilon}, \overline{q})$, where $\overline{\varepsilon} = (\varepsilon_1, \ldots, \varepsilon_k)$ $\overline{q} = (q_1, \ldots, q_k)$.
Let $N(\overline{\varepsilon},\overline{q})$ be the integer number  resulting in  performing all the operations in (\ref{eq:bin_decomp}).

The standard binary representation of a natural number  is a binary sum with $\varepsilon_i \in \{0,1\}$. Every integer can be  represented by infinitely many different binary sums. We say that two binary sums are  {\em equivalent} if they represent the same number.
 Furthermore, a
binary sum $P(\overline{\varepsilon}, \overline{q})$ is {\em
reduced} if  the sequence $\overline{q}$ is strictly decreasing.
 The following lemma is obvious.

\begin{lemma} \label{le:unique_positive_bin_sum}
The following hold:
\begin{itemize}
    \item [1)]
For each binary sum
$P(\overline{\varepsilon}, \overline{q})$ there exists an
equivalent reduced binary sum which can be computed in linear time
$O(|\overline{q}|)$.
    \item[2)]
For any positive integer $z$ there exists a unique reduced binary
sum $P(\overline{\varepsilon}, \overline{q})$ with  $\varepsilon_1 = \ldots =
\varepsilon_k = 1$ and  $q_k
\le \lfloor\log_2 z \rfloor$, representing $z$.
Furthermore, it can be found in $O(\log_2 z)$ time.
\end{itemize}
\end{lemma}

The unique binary sum representing a given natural number $N$ with all coefficients $\varepsilon_i  = 1$ is called {\em positive normal form} of $N$.

\begin{remark}
Notice, that positive binary representations of numbers   may not be the most efficient. For instance, the
binary sum $2^n-2^0$ is equivalent to $2^{n-1} + 2^{n-2} + \ldots
+ 2^{1} + 2^{0}$ but has much fewer terms.
\end{remark}

\begin{lemma} \label{le:triv_binary} \label{le:positive_reduced_binary}
\label{le:binary_divisibility}
Let $P(\overline{\varepsilon}, \overline{q})$ be a reduced
binary sum. Then:
\begin{enumerate}
\item [1)]  $N(\overline{\varepsilon}, \overline{q}) =
0$ if and only if $|\overline{q}| = 0$ (here $|\overline{q}|$ is the length of the tuple $\overline{q}$).
\item [2)] $N(\overline{\varepsilon}, \overline{q}) > 0$ if and
only if $\varepsilon_1=1$.
\item [3)]  $N(\overline{\varepsilon},
\overline{q})< 0$ if and only if $\varepsilon_1=-1$.
\item[4)] $N(\overline{\varepsilon}, \overline{q})$ is divisible
by $2^n$ if and only if $q_m \ge n$ (here  $m = |\overline{q}|$, and $n \in \mathbb{N}$).
 \item[5)] In the notation above if $N(\overline{\varepsilon}, \overline{q})$ is divisible
by $2^n$ and not divisible by $2^{n+1}$ then $q_m = n$.
\end{enumerate}
\end{lemma}

\begin{proof} We prove 1), the rest is similar.
If $|\overline{q}| = 0$ then $N(\overline{\varepsilon},
\overline{q}) = 0$. Assume now that $N(\overline{\varepsilon},
\overline{q}) = 0$ and $\overline{q} = (q_1,\ldots,q_k)$, where
$k>0$. Let $S = \{ 1\le i \le k \mid \varepsilon_i>0 \}$. Then
$$N(\overline{\varepsilon}, \overline{q}) = \left( \sum_{i\in S} 2^{q_i} \right) - \left( \sum_{j \in \{1,\ldots, k\} \setminus S} 2^{q_j}\right).$$
The binary sums in the brackets have
coefficients $1$. Since $P(\overline{\varepsilon}, \overline{q})$
is reduced these binary sums are different and by
Lemma \ref{le:unique_positive_bin_sum}  define different
numbers. This implies that  $N(\overline{\varepsilon}, \overline{q}) \ne 0$, and 1) follows by contradiction.
\end{proof}

Let $P(\overline{\varepsilon}, \overline{q})$ be a reduced binary
sum. We say that a pair of powers $(q_{i},q_{i+1})$ in
$P(\overline{\varepsilon}, \overline{q})$ is {\em superfluous} if
$q_{i}=q_{i+1}+1$ and $\varepsilon_{i} = -\varepsilon_{i+1}$. The
next lemma shows that a binary sum with superfluous pairs can be
simplified, by getting rid off such pairs in linear time.

\begin{lemma} \label{le:superfluous}
Given a binary sum $P(\overline{\varepsilon}, \overline{q})$ one can find an equivalent reduced
binary sum without superfluous pairs in liner time $O(|\overline{q}|)$.
\end{lemma}

\begin{proof}
Let $(q_i,q_{i+1})$ be a superfluous pair in
$P(\overline{\varepsilon}, \overline{q})$. Define
$$\overline{q}' = (q_1 , \ldots , q_{i-1},q_{i+1}, \ldots, q_n)$$
and
$$\overline{\varepsilon}' = ( \varepsilon_1 , \ldots , \varepsilon_{i-1},-\varepsilon_{i+1}, \ldots, \varepsilon_n).$$
The equality $\mp 2^{i+1} \pm 2^{i} = \mp 2^{i}$  implies  $N(\overline{\varepsilon}, \overline{q}) =
N(\overline{\varepsilon}', \overline{q}')$. Clearly,  it requires linear number  (in $|\overline{q}|$) of steps like that to eliminate all superfluous pairs in $P(\overline{\varepsilon}, \overline{q})$.
\end{proof}

  For a reduced
binary sum $\CP(\Lq,\Le)$ define
$$\varepsilon(\CP,q)
 = \left\{
\begin{array}{ll}
\varepsilon_j, & \mbox{if there exists (unique)  $j$ such that} \ q_j=q;\\
0,             & \mbox{otherwise.}
\end{array}
\right.
$$
The following technical lemma gives the main tool for efficiently comparing values of binary sums.

\begin{lemma} \label{le:comp_pos_binary_decomp}
Let $A = P(\Le,\Lq)$ and $B = P(\Ld,\Lr)$ be reduced binary sums
without superfluous pairs, $k = |\Lq|$, and $m = |\Lr|$. Put $n =
\max\{q_1,r_1\}$, $\alpha_1 = \varepsilon(A,n)$, $\alpha_2 =
\varepsilon(A,n-1)$, $\beta_1 = \varepsilon(B,n)$, and $\beta_2 =
\varepsilon(B,n-1)$. Then the  following hold:
\begin{enumerate}
    \item[1)]
If $\alpha_1 = 1$ and $\beta_1 = -1$ then $N(A)-N(B)\ge
2$. Similarly, if $\alpha_1 = -1$ and $\beta_1 = 1$ then
$N(A)-N(B) \le 2$.
    \item[2)]
Assume $\alpha_1 = 1$ and $\beta_1 = 1$, or $\alpha_1 = -1$ and
$\beta_1 = -1$. Define $A' = P(\Le',\Lq')$ and $B' =
P(\Ld',\Lr')$, where $\Lq' = (q_2,\ldots,q_k)$, $\Le' =
(\varepsilon_2,\ldots,\varepsilon_k)$, $\Lr' = (r_2,\ldots,r_m)$,
$\Ld' = (\delta_2,\ldots,\delta_m)$. Then $N(A)-N(B) = N(A') -
N(B')$.
    \item[3)]
Assume $\alpha_1=1$ and $\beta_1=0$:
\begin{enumerate}
    \item[a)]
If $\alpha_2 = 1$ then $N(A)-N(B)\ge 2$.
    \item[b)]
If $\alpha_2 = 0$ and $\beta_2 <1$ then $N(A)-N(B)\ge 2$.
    \item[c)]
If $\alpha_2 = 0$ and $\beta_2 =1$ define $A' = P(\Le',\Lq')$ and
$B' = P(\Ld',\Lr')$, where $\Lq' = (n-1,q_2,\ldots,q_k)$, $\Le' =
(1, \varepsilon_2, \ldots, \varepsilon_k)$, $\Lr' =
(r_2,\ldots,r_m)$, $\Ld' = (\delta_2,\ldots,\delta_m)$. Then
$N(A)-N(B) = N(A') - N(B')$.
\end{enumerate}
\end{enumerate}
\end{lemma}

\begin{proof}
In the case 1) by Lemma \ref{le:positive_reduced_binary} $N(A)
\ge 1$ and $N(B) \le -1$, so the statement holds. In the case 2)
 $N(A') = N(A) - 2^{n}$ and $N(B') = N(B) -
2^{n}$ and the statement holds. In the case 3.a)
 $N(A) \ge 2^{n} +2^{n-1}-2^{n-2}+1$ and $N(B) \le 2^{n} +
2^{n-1} -1$ (since  $A$ and $B$ have no
superfluous pairs). In the case 3.b) $N(A) \ge 2^{n-1}+1$ and $N(B)
\le 2^{n-1}-1$. In the  case 3.c) $N(A') = N(A)
- 2^{n-1}$ and $N(B') = N(B) - 2^{n-1}$.  These imply that 3) holds.
\end{proof}

\begin{proposition} \label{pr:complexity_compare}
For given binary sums $P(\overline{\varepsilon}, \overline{q})$ and
$P(\overline{\delta}, \overline{r})$ it takes linear time $C(|\overline{q}|+|\overline{r}|)$ to compare
the values $N(\overline{\varepsilon}, \overline{q})$ and
$N(\overline{\delta}, \overline{r})$.
\end{proposition}

\begin{proof}
By Lemmas \ref{le:unique_positive_bin_sum} and \ref{le:superfluous}
one can reduce and get rid off superfluous pairs in given binary sums in linear time.
Now, let $A = P(\Le,\Lq)$ and $B = P(\Ld,\Lr)$ be reduced binary sums
without superfluous pairs. In the notation of Lemma \ref{le:comp_pos_binary_decomp} one can describe the comparison algorithm as follows. Determine the
values $\alpha_1, \alpha_2$ and $\beta_1$, $\beta_2$. If they satisfy either of the case 1,  3.a, or 3.b
 then the answer follows immediately from the  lemma. Otherwise, they satisfy either the case 2 or
3.c, and one  can compute new binary sums $A'$ and $B'$ such that
$N(A')-N(B') = N(A)-N(B)$ and $|A'|+|B'| < |A|+|B|$, and compare
their values.   Notice  that  the binary sum
$A'$ in case 3.c) might contain a superfluous pair, which should be removed in the simplification process.

Now we describe the comparison algorithm formally.

\begin{algorithm} \label{al:compare_binaries}({\em To compare values of reduced
binary sums with no superfluous pairs.})
    \\{\sc Input.}
$P(\overline{\varepsilon}, \overline{q})$ and
$P(\overline{\delta}, \overline{r})$ two reduced binary sums with
no superfluous pairs of powers.
    \\{\sc Output.}
$$
\left\{
\begin{array}{rl}
-2, & \mbox{if } N(\overline{\varepsilon},
\overline{q})< N(\overline{\delta}, \overline{r})-1 \\
-1, & \mbox{if } N(\overline{\varepsilon},
\overline{q})= N(\overline{\delta}, \overline{r})-1 \\
0, & \mbox{if } N(\overline{\varepsilon},
\overline{q})= N(\overline{\delta}, \overline{r}) \\
1, & \mbox{if } N(\overline{\varepsilon},
\overline{q})= N(\overline{\delta}, \overline{r})+1 \\
2, & \mbox{if } N(\overline{\varepsilon},
\overline{q})> N(\overline{\delta}, \overline{r})+1 \\
\end{array}
\right.
$$
{\sc Computations.}
\begin{enumerate}
    \item[A)]
Remove all superfluous pairs  from
$P(\overline{\varepsilon}, \overline{q})$ and
$P(\overline{\delta}, \overline{r})$.
    \item[B)]
Compute $n = \max\{q_1,r_1\}$.
    \item[C)]
If $n\le 1$ then the current binary sums
$P(\overline{\varepsilon}, \overline{q})$ and
$P(\overline{\delta}, \overline{r})$ are at most one-bit numbers.
Compute them, compare, and output the result.
    \item[D)]
If $n>1$ then compute $\alpha_1 = \varepsilon(P(\Le,\Lq),n)$,
$\alpha_2 = \varepsilon(P(\Le,\Lq),n-1)$, $\beta_1 =
\varepsilon(P(\Ld,\Lr),n)$, and $\beta_2 =
\varepsilon(P(\Ld,\Lr),n-1)$.
    \item[E)]
Determine if $(\alpha_1,\alpha_2)$ and $(\beta_1,\beta_2)$ satisfy one of
the cases 1, 3.a, or 3.b from Lemma
\ref{le:comp_pos_binary_decomp}. If so, return the result
prescribed in Lemma.
    \item[F)]
Determine if $(\alpha_1,\alpha_2)$ and $(\beta_1,\beta_2)$ satisfy one of
the cases 2 or 3.c. If so, compute new binary sums $A'$ and
$B'$ as prescribed in Lemma \ref{le:comp_pos_binary_decomp} put
$P(\Le,\Lq) = A'$ and $P(\Ld,\Lr) = B'$ and goto A).
\end{enumerate}
\end{algorithm}

Notice, that each iteration of Algorithm \ref{al:compare_binaries} decreases
the number $|\overline{q}|+|\overline{r}|$ at least  by $1$, so the algorithm terminates in at most   $C(|\overline{q}|+|\overline{r}|)$ steps, as claimed.

\end{proof}

\subsection{Shortest binary forms}
\label{sec:compact-forms}

Let $P(\Le,\Lq)$ be a reduced binary sum, where
$\Lq = (q_1,q_2,\ldots,q_k)$, and $\Le = (\varepsilon_1,\ldots,\varepsilon_k)$.
We say that $P(\Le,\Lq)$ is {\em compact} if $q_{i+1}-q_i \ge 2$ for every $i=1,\ldots,k-1$.

\begin{lemma}\label{le:compact_sum}
The following hold:
\begin{enumerate}
    \item[(1)]
For any $n\in\MN$ there exists  a unique compact binary sum
$P_n = \varepsilon_1 2^{q_1} + \ldots +  \varepsilon_k 2^{q_k}$
representing $n$. Furthermore, $k, q_1, \ldots, q_k \leq \log_2n$ and
$P_n$ can be found in linear time $O(\log_2 n)$.
    \item[(2)]
A compact binary sum representation of a given number involves
the least possible number of terms.
    \item[(3)]
Given a binary sum one can find an equivalent compact binary sum in linear time.
\end{enumerate}
\end{lemma}

\begin{proof}
By Lemma \ref{le:unique_positive_bin_sum} for $n\in\MN$ we can find a reduced
binary sum $\CP$ representing $n$ in time $O(\log_2 n)$. Below we prove the existence
and uniqueness of a compact binary sum equivalent to $\CP$.

{\bf Existence.} Consider any binary sum $P(\Le,\Lq)$.
Consider the following finite rewriting system $\mathcal{C}$ on binary sums: a system of transformations of binary sums
$$
\left\{
\begin{array}{l}
2^{m}+2^{m} \rightarrow 2^{m+1}\\
2^{m}-2^{m} \rightarrow \varepsilon\\
2^{m+1} +2^{m} \rightarrow 2^{m+2}-2^{m}\\
2^{m+1} -2^{m} \rightarrow 2^{m}\\
\end{array}
\right.
$$
Obviously, each application of a rule from $\mathcal{C}$ to a binary sum results in an equivalent binary sum, which is either shorter or has the same length as the initial sum.
It is easy  to see that the system $\mathcal{C}$ is terminating, i.e., starting on a given binary sum $P(\Le,\Lq)$ after finitely many steps of rewriting one arrives to a sum that no rule from $\mathcal{C}$ can be applied to. Observe,  that the number of steps required here is at most linear in the length of $P(\Le,\Lq)$. Furthermore,
the system  $\mathcal{C}$ is locally confluent, hence confluent (see \cite{book93} for definitions). This implies that the rewriting of a given binary sum always results in a compact form and such a form does not depend on the rewriting process. In particular, applying the rewriting process to the standard binary representation of a given natural number  $n$ one can find the shortest binary form of $n$ (and of $-n$) in linear time.

{\bf Uniqueness.}
Consider two compact binary sums
$$P(\Le,\Lq) = \sum_{i=1}^k \varepsilon_i 2^{q_i}, \ \ \  P(\Ld,\Lp) = \sum_{i=1}^s \delta_i 2^{p_i}.$$
Observe that
\begin{itemize}
    \item
If $\varepsilon_k\ne \delta_s$ then $N(\Le,\Lq)$ and $N(\Ld,\Lp)$
have opposite signs, in particular $N(\Le,\Lq) \ne N(\Ld,\Lp)$.
    \item
If $\varepsilon_k=\delta_s = 1$ and $q_k > p_s$ then
    $$N(\Le,\Lq)-N(\Ld,\Lp) \ge (2^{q_k}-2^{q_k-2}-2^{q_k-4}-\ldots)-(2^{q_k-1}+2^{q_k-3}+2^{q_k-5}+\ldots) \ge 1.$$
In particular $N(\Le,\Lq) \ne N(\Ld,\Lp)$.
    \item
Similarly,  $N(\Le,\Lq) \ne N(\Ld,\Lp)$
whenever $\varepsilon_k=\delta_s = -1$ and/or $q_k < p_s$.
\end{itemize}
Therefore, equality $N(\Le,\Lq) = N(\Ld,\Lp)$ implies that
$\varepsilon_k=\delta_s$ and $q_k = p_s$. Using this  it is easy
to prove that two compact binary sums representing the same number
are equal.

{\bf Minimality.}
Any non-compact binary sum can be rewritten into an equivalent compact binary sum
by the length non-increasing system $\mathcal{C}$.
Therefore, the compact binary sums involve the least
possible number of terms.
\end{proof}

\begin{lemma}\label{le:binary_to_compact}
Suppose $P(\Le,\Lq)$ is reduced and $P(\Ld,\Lp)$ is the equivalent compact binary sum.
Then for every $d\in \Lp$ either $d\in \Lq$ or $d-1 \in \Lq$.
Furthermore, if $N(\Le,\Lq)\ne 0$ then the compact binary sum representing the number $N(\Le,\Lq)+1$
satisfies the same condition.
\end{lemma}

\begin{proof}
We may assume that $P(\Le,\Lq)$ does not contain superfluous pairs
because removing superfluous pairs from $P(\Le,\Lq)$ results in a new
binary sum $P(\Le',\Lq')$ where $\Lq' \subseteq \Lq$.
Therefore there exist sequences of positive integers $\{a_i\}$, $\{b_i\}$
and a sequence $\{\varepsilon_i\}$
such that
    $$P(\Le,\Lq) = (\varepsilon_1 2^{a_1}+\ldots +\varepsilon_12^{a_1+b_1}) + \ldots + (\varepsilon_k 2^{a_k}+\ldots +\varepsilon_k2^{a_k+b_k})$$
where $a_i+b_i < a_{i+1}$ and $\varepsilon_i = \pm 1$.
Making the sum in the first brackets compact we get
    $$P(\Le,\Lq) = (-\varepsilon_1 2^{a_1}+\varepsilon_1 2^{a_1+b_1+1}) + \ldots + (\varepsilon_k 2^{a_k}+\ldots +\varepsilon_k2^{a_k+b_k}).$$
If $a_1+b_1+1 \le a_2-2$ then we can think that $(-\varepsilon_1 2^{a_1}+\varepsilon_1 2^{a_1+b_1+1})$ is already compact
and consider the next sum. The induction finishes the proof in this case.

Assume that $a_1+b_1+1 = a_2-1$.
If $\varepsilon_1 = -\varepsilon_2$ then $\varepsilon_1 2^{a_1+b_1+1}+\varepsilon_2 2^{a_2}$ is a superfluous pair.
Removing it we obtain
    $$P(\Le,\Lq) = (-\varepsilon_1 2^{a_1}-\varepsilon_1 2^{a_1+b_1+1}) + (\varepsilon_2 2^{a_2+1}+\ldots +\varepsilon_22^{a_2+b_2}) + \ldots$$
and as above the sum $(-\varepsilon_1 2^{a_1}-\varepsilon_1 2^{a_1+b_1+1})$ is compact and we can consider the next sum.
If $\varepsilon_1 = \varepsilon_2$ then the power $\varepsilon_1 2^{a_1+b_1+1}$
is being added to the second sum. Induction finishes the proof in this case.

Observe that in each case either we do not introduce a new power of $2$ or we stop at $2^{a_1+b_1+1}$.
Therefore, for every $d\in \Lp$ either $d\in \Lq$ or $d-1 \in \Lq$.
In a similar way we can prove the last statement of the lemma.
\end{proof}

\section{Power circuits}
\label{se:power_circuits}

We gave a definition of general algebraic circuits in the language
$\mathcal{L} = \{+,-,\cdot, x\cdot2^y\}$ in the introduction.
In this section we define a special type of circuits,
called {\em power circuits}. Power circuits are main technical objects of the paper.
They can be viewed as versions of the algebraic circuits of a special kind.
We show in due course that every algebraic circuit in $\CL$ is equivalent in
the structure $\tilde Z = \langle \mathbb{Z}, +,-,\cdot, x\cdot2^y\rangle$
to a power circuit, but power circuits are much easier to work with.
Besides, power circuits give a very compact presentation of natural
numbers, designed specifically for
efficient computations with exponential polynomials.

\subsection{Power circuits and terms}

Let $\CP = (V(\CP),E(\CP))$ be a directed graph. For an edge $e =
v_1 \rightarrow v_2 \in E(\CP)$ we denote by $\alpha(e)$ its {\em
origin} $v_1$ and by $\beta(e)$ its {\em terminus} $v_2$. We say
that $\CP$ contains {\em multiple edges} if there are two distinct
edges $e_1$ and $e_2$ in $\CP$ such that $\alpha(e_1) =
\alpha(e_2)$ and $\beta(e_1) = \beta(e_2)$. For a vertex $v$ in
$\CP$ denote by $Out_v$ the set of all edges with the origin $v$
and by $In_v$ the set of all edges with the terminus $v$. A vertex $v$ with
$Out_v = \emptyset$  is called a {\em leaf} or a {\em gate}; $Leaf(\CP)$ is the set of leaves in $\CP$.

A {\em power circuit} is a tuple  $(\CP, \mu, M, \nu, \gamma)$ where:
\begin{enumerate}
    \item[$\bullet$]
$\CP = (V(\CP),E(\CP))$ is a non-empty directed acyclic graph  with no multiple edges;
    \item[$\bullet$]
$\mu:E(\CP) \rightarrow  \{ 1,-1\}$ is called the  {\em
edge labeling function};
    \item[$\bullet$]
$M \subseteq V(\CP)$ is a non-empty subset of vertices called
{\em the marked vertices};
    \item[$\bullet$]
 $\nu:M \rightarrow \{-1,1\}$ is called a {\em  sign
function}.
\item[$\bullet$]  $\gamma: Leaf(\CP) \to X \cup\{0\}$ is a function which assigns to each leaf in $\CP$ either a variable from a set of variables $X$ or the constant   $0$.
\end{enumerate}
For simplicity we often omit $\mu, M, \nu, \gamma$ from notation and refer to the power circuit above as $\CP$.

For a power circuit $\CP$ we  define a term $t_v$ in the language $\CL$  for each  vertex $v \in V(\CP)$,  by induction starting at leaves (which exists since $\CP$ is acyclic):
$$
t_v = \left\{
\begin{array}{ll}
\gamma(v) & \mbox{if } v \in Leaf(\CP); \\
2^{\sum_{e\in Out_v} \mu(e)t_v(\beta(e))} & \mbox{otherwise.}\\
\end{array}
\right.
$$
where the sum $\sum_{e\in Out_v}$ denotes  composition of additions in some fixed order on terms in $\CL$.
Finally, define a term
    $$\CT_\CP = \sum_{v\in M} \nu(v) t_v.$$

The number $|\CP| = |V(\CP)| + |E(\CP)|$ is called the {\em size} of a power circuit $\CP$.  Two circuits
$\CP_1$ and $\CP_2$  are  {\em equivalent} (symbolically $\CP_1 \sim \CP_2$) if the
terms $\CT_{\CP_1}$ and $\CT_{\CP_2}$ induce the same function in $\tilde Z$.
Notice, that these functions could be partial (not everywhere defined) on $\tilde Z$.

\subsection{Term evaluation and constant circuits}

A power circuit $\CP = (\CP, \mu, M, \nu, \gamma)$ is called {\em constant}
if the function $\gamma$ assigns no variables (i.e., $\gamma\equiv 0$).
In this case every term $t_v$ ($v \in V(\CP)$),
represents a real number which we denote by $\CE(v)$. The real represented
by  $\CT_\CP$ is denoted by $\CE(\CP)$. We say that $\CP$
{\em properly represents} an integer  number $N$ if $N = \CE(\CP)$
and $\CE(v) \in \mathbb{N}$ for every
$v \in V(\CP)$. In this case we write  $N = \CE(\CP)$.
Notice that the term $\CT_\CP(X)$ is defined in $\tilde Z$
for an assignment of variables $\eta : X \to \mathbb{Z}$
if and only $\CP$ properly  represent an integer $N(\CP)$.
Similarly, we say that $\CP$ properly represents a natural number if $\CE(\CP) \in\MN$ and
$\CE(v) \in \mathbb{N}$ for every $v \in V(\CP)$.
Observe, that  two constant circuits $\CP_1$ and $\CP_2$ are {\em equivalent}
if $\CE(\CP_1) = \CE(\CP_2)$.
For constant circuits we omit the function $\gamma$ from notation.
Equivalent power circuits $\CP_1$ and $\CP_2$ are {\em strongly equivalent} if $\CP_1$ is proper if and only if $\CP_2$ is proper.

\begin{lemma} \label{le:p-circuit-number}
For $n\in\MN$ one construct a power circuit $\CP$ properly representing  $n$
in time $O(\log^2_2n)$.
\end{lemma}

\begin{proof}
Induction on $n$.
By Lemma \ref{le:unique_positive_bin_sum} for a given number $n$ one can find
in  time $O(\log_2n)$ the a reduced binary sum
$\varepsilon_1 2^{q_1} + \ldots +  \varepsilon_k 2^{q_k}$ representing $n$,
where $k, q_1, \ldots, q_k \leq \log_2n$.
By induction, one can construct power circuits $C_1, \ldots, C_k$ representing the numbers $q_1, \ldots,q_k$.
It takes time $O(\log_2^2(\log_2n))$ time to construct each circuit $C_i$.
So altogether it takes at most $O(\log_2n\log_2^2(\log_2n))$ time.
Given power circuits $C_1, \ldots, C_k$ it requires additional time $O(\log_2n)$
to a construct a power circuit representing $n$.
The time estimate follows from the obvious  observation
$log_2n\log_2^2(\log_2n) = O(\log_2^2n).$
\end{proof}

See Figure \ref{fi:pp_examples} for examples of (constant) power circuits. In
figures we denote unmarked vertices by white circles and marked
vertices by black circles. Each edge and marked vertex is labelled
with the plus or minus sign denoting $1$ or $-1$ respectively.

\begin{figure}[htbp]
\centerline{ \includegraphics[scale=0.5]{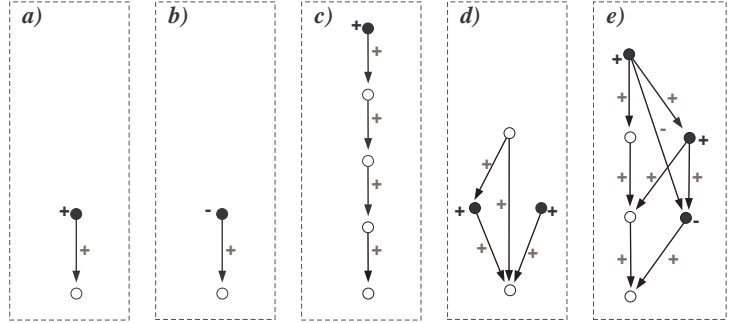} }
\caption{\label{fi:pp_examples} Examples of power circuits
representing integers $1$, $-1$, $16$, $2$, and $35$.}
\end{figure}

Let $\CP$ be a power circuit, $\CT_\CP  = \CT_\CP(X)$, and $\eta:X \to \mathbb{Z}$  an
assignment of variables in $X$. The following lemma allows one to
operate with the value $\CT_\CP(\eta(X))$ by means of constant power circuits.

\begin{lemma}\label{le:substitution}
Let $\CP$ be a power circuit, $\CT_\CP  = \CT_\CP(X)$,
and $\eta:X \to \mathbb{Z}$  an  assignment of
variables. Then one can construct a constant power circuit $\CP'$
representing the number $\CT_\CP(\eta(X))$ in time
$O(|\CP|+\log_2^2(size(\eta)))$, where $size(\eta) = \sum_{x \in X} \log_2 (\eta(x))$.
Moreover, the value  $\CT_\CP(\eta(X))$ is defined in $\tilde Z$
if and only if the circuit $\CP'$  properly represents $\CT_\CP(\eta(X))$.
\end{lemma}

\section{Standard, reduced and normal power circuits}
\label{se:circuit_types}

In this section we define several important types of circuits: standard, reduced and normal. The standard ones can be  easily obtained from general power circuits through some obvious simplifications. The reduced power circuits output numbers only, they require much stronger rigidity conditions (no redundant or superfluous pairs of edges, distinct vertices output distinct numbers), which are much harder to achieve. The normal power circuits are reduced and output numbers in the compact binary forms. They give a unique compact  presentation of integers, which is much more compressed (in the worst case) than the canonical binary representations. This is the main construction of the paper, it is interesting in its own right.

\subsection{Standard circuits}
\label{se:standard-circuits}

We say that a vertex $v$ is a {\em zero vertex} in a circuit $\CP$ if  $v \in Leaf(\CP)$ and $\gamma(v) = 0$. If $\CP$ is a constant circuit then  $v$  is a zero vertex if and only if $\CE(v) = 0$.
The following lemma is obvious.

\begin{lemma}\label{le:zero_exists}
A constant power circuit contains at least one zero vertex.
\end{lemma}

Let $\CP$ be a constant power circuit. Below we describe some obvious {\em rewriting rules} that allow one to simplify  $\CP$ (if applicable) keeping the strong equivalence.

\medskip
\noindent {\bf Trivializing:} Notice that  if every marked vertex in $\CP$ is a zero vertex then $\CE(\CP) = 0$. In this event we replace $\CP$ by a strongly equivalent circuit $\CP'$ consisting of a single marked vertex $v$.

From now on we assume that $\CP$ has a non-zero marked vertex.

\medskip
\noindent {\bf Unmark a zero:}
\label{le:marked_zero}
Let  $v$ be a marked zero vertex in $\CP$. If  $\CP'$ is obtained from $\CP$ by making $v$ unmarked then
$\CP$ and $\CP'$ are strongly equivalent.

\medskip
\noindent {\bf Fold two zeros:}
\label{le:many_zeros}
Let  $v_1,v_2\in V(\CP)$ be two zeros in $\CP$.
 If $\CP'$ is  obtained from $\CP$ by
folding $v_1$ and $v_2$ then  $\CP$ and $\CP'$ are strongly equivalent.

\medskip
\noindent {\bf Remove redundant ``zero edges'':}
 \label{le:zero_edges}
Let $z$ be a zero vertex in $\CP$, $e =
v\rightarrow z \in E(\CP)$, and $|Out_{v}|>1$. If $\CP'$ is
obtained from $\CP$ by removing the edge $e$ then  $\CP$ and $\CP'$ are strongly equivalent.

\begin{figure}[htbp] \centerline{
\includegraphics[scale=0.5]{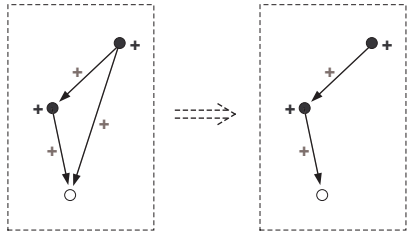} }
\caption{\label{fi:redundantzeroedge} Removing redundant zero edges.}
\end{figure}

\label{se:trimmed_reduced_graph_pres}

A power circuit $\CP$ is {\em trimmed} if for  each vertex $v \in \CP$ there ia a directed path from a marked
vertex to $v$. The following rewriting rule allows one to trim circuits.

\medskip
\noindent
{\bf Trimming:} Let $v$ be an unmarked vertex $v \in \CP$ with $In_v =
\emptyset$. If $\CP'$ is obtained from $\CP$ by removing
the vertex $v$ and  all the adjacent edges then $\CP'$ is strongly equivalent to $\CP$.

\begin{figure}[htbp]
\centerline{
\includegraphics[scale=0.5]{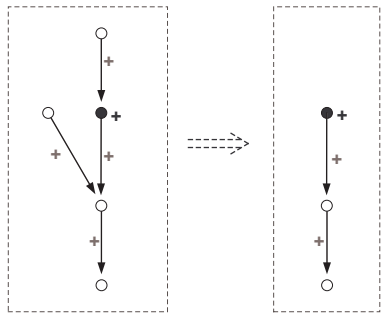} }
\caption{\label{fi:trim_examples} Example of trimming.}
\end{figure}

\begin{definition}
A trimmed power circuit $\CP$ is in the {\em standard} form if it contains a unique
unmarked zero vertex and contains no redundant zero edges.
\end{definition}

\begin{algorithm} \label{al:remove_redundancies}({\em Standard power circuit})
\label{al:trim}
     \\{\sc Input.}
A circuit $\CP$.
    \\{\sc Output.}
A strongly equivalent circuit $\CP'$ in a standard form.
    \\{\sc Computations.}
\begin{enumerate}
    \item[(1)]
Compute the set $Leaf(\CP)$.
    \item[(2)]
Fold all zero vertices in $Leaf(\CP)$ into one vertex $z$ and make it unmarked.
    \item[(3)]
Erase all redundant edges incoming into $z$.
    \item[(4)]
Trim the circuit.
    \item [(5)]
Return the result $\CP'$.
\end{enumerate}
\end{algorithm}

Summarizing the argument above one has the following result.
\begin{proposition} \label{pr:red_zeros_equivalence} \label{pr:trimming}
Let $\CP'$ be produced from $\CP$ by Algorithm \ref{al:remove_redundancies}. Then
\begin{itemize}
    \item
$\CP'$ is standard and is strongly equivalent to $\CP$.
    \item
$|V(\CP)|\leq |V(\CP')|$ and $|E(\CP)|\leq |E(\CP')|$.
    \item
it takes time $O(|\CP|)$ to construct $\CP'$.
\end{itemize}
\end{proposition}

There is one more procedure that is useful for operations over
power circuits (see Sections \ref{se:multiplication} and \ref{se:multiplication_power}).
Recall that a vertex $v$ in $\CP$ is a {\em source}
if $In_v = \emptyset$. The following algorithm converts a circuit
into an equivalent one where each marked vertex is a source.

\begin{algorithm}\label{al:marked_origins}\
    \\{\sc Input.}
A circuit $\CP = (\CP,M,\mu,\nu)$.
    \\{\sc Output.}
An  equivalent circuit $\CP'$ in which every marked vertex is a source.
    \\{\sc Computations:}
\begin{enumerate}
\item[A.] For each vertex $v \in M$ with  $In_v
\ne \emptyset$ do:
\begin{enumerate}
    \item[(1)]
introduce a new vertex $v'$;
    \item[(2)]
for each edge $v \stackrel{\varepsilon}{\rightarrow} u$ introduce
a new  edge $v' \stackrel{\varepsilon}{\rightarrow} u$;
    \item[(3)]
replace $v$ with $v'$ in $M$ and put $\nu(v') = \nu(v)$.
\end{enumerate}
    \item[B)]
Output the obtained circuit.
\end{enumerate}
\end{algorithm}

\begin{lemma} \label{le:source}
Let $\CP'$ be produced by Algorithm \ref{al:marked_origins} from $\CP$. Then:
 \begin{itemize}
 \item $\CP$ is strongly equivalent to $\CP'$,
 \item $|V(\CP')| \le 2|V(\CP)|$, and $|E(\CP')| \le 2|E(\CP)|$.
 \item Algorithm \ref{al:marked_origins} has linear time complexity
$O(|\CP|)$.
\end{itemize}
\end{lemma}

\begin{figure}[htbp]
\centerline{
\includegraphics[scale=0.5]{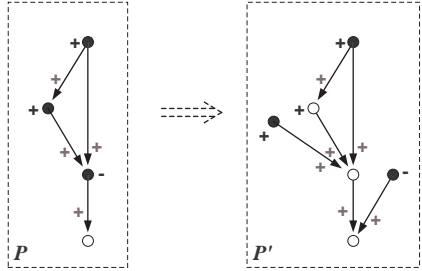} }
\caption{\label{fi:marked_vert_separation} Processing of marked vertices that are not sources.}
\end{figure}

\subsection{Reduced power circuits}

Let $\CP$ be a constant power circuit in the standard form.
A pair of  edges  $e_1 = v\rightarrow v_1$ and  $e_2 =
v\rightarrow v_2$  with the same origin $v$ is called a {\em redundant } pair in $\CP$  if $\mu(e_1) = -\mu(e_2)$ and
$\CE(v_1) = \CE(v_2)$.

\medskip
\noindent
{\bf Removing redundant edges:}
Let $e_1$ and $e_2$ be a redundant pair of edges in $\CP$. If $\CP'$ is
obtained from $\CP$ by removing the pair $e_1, e_2$ then $\CP'$ is
equivalent to $\CP$. Moreover, if $\CP$ properly represents an integer
then $\CP'$ properly represents the same integer.

 A pair of edges  $e_1 = v\rightarrow v_1$ and  $e_2 =
v\rightarrow v_2$ as above is termed {\em superfluous} if
$\mu(e_1) = -\mu(e_2)$ and $\CE(v_1) = 2\CE(v_2)$.

\medskip
\noindent
{\bf Removing superfluous edges:}
Let $(e_1,e_2)$ be a pair of superfluous edges in $\CP$. If  $\CP'$ is
obtained from $\CP$ by removing  the edge $e_1$ from $\CP$ and changing
$\mu(e_2)$ to $-\mu(e_2)$ then $\CP'$ is strongly equivalent to $\CP$.

\medskip
\noindent
{\bf Remark.}  If one knows what pairs of edges are redundant or
superfluous in $\CP$ then it takes time $O(|\CP|)$ to remove
them (applying the rules above).  However, it is not obvious how to
check efficiently if $\CE(v_1) = \CE(v_2)$ or $\CE(v_1) = 2\CE(v_2)$.
We take care of this in due course.

\begin{figure}[htbp] \centerline{
\includegraphics[scale=0.5]{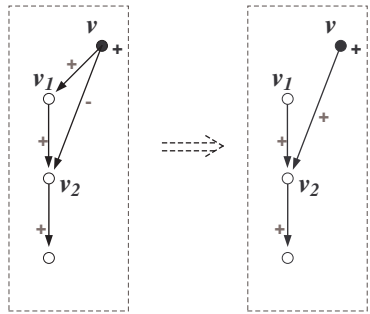} }
\caption{\label{fi:lemma49} Removing superfluous edges. Here, $\CE(v_1) = 2\CE(v_2) =
2$.}
\end{figure}

\begin{definition}
A circuit $\CP$ is {\em reduced} if
\begin{enumerate}
    \item[(R1)]
$\CP$ is in the standard form.
    \item[(R2)]
For any $v_1, v_2 \in V(\CP)$, $\CE(v_1)=\CE(v_2)$ if and only if
$v_1=v_2$.
    \item[(R3)]
$\CP$ contains no redundant or superfluous edges.
\end{enumerate}
\end{definition}

\begin{proposition}\label{pr:circuit_compare}
Let $\CP$ be a reduced circuit. Then
\begin{enumerate}
 \item[1)] $\CE(\CP)=0$ if and only if $\CP$ is trivial, i.e., $\CP$ consists of a single marked vertex.
 \item[2)] Let $v \in M$ be such that for any $v' \in M$ $\CE(v) \ge
 \CE(v')$ (the vertex with the maximal $\CE$-value in $M$). Then:
  \begin{itemize}
  \item If $\nu(v)=1$ then $\CE(\CP)>0$.
  \item If $\nu(v)=-1$ then $\CE(\CP)<0$.
 \end{itemize}
\end{enumerate}
\end{proposition}

\begin{proof}
If $\CP$ is trivial then clearly $\CE(\CP)=0$. Now, suppose
that $\CP$ is not
trivial. Then $\CE(\CP) = \sum_{v\in M} \nu(v) \CE(v)$, where $\CE(v) =
2^{\sum_{e\in Out_v} \mu(e)\CE(\beta(e))}$. Hence  $\CE(\CP)$
is a reduced binary sum, so by Lemma \ref{le:triv_binary}, it is not equal to $0$. Which proves 1).

The second statement can be proved similarly using Lemma
\ref{le:positive_reduced_binary}.

\end{proof}


\subsection{Normal forms of constant power circuits}

Let $\CP$ be a constant power circuit.
We say that $\CP$ is in the {\em normal form} if
\begin{itemize}
    \item[(N1)]
$\CP$ is proper and reduced.
    \item[(N2)]
For every vertex $v \in V(\CP)$ the binary sum $\sum_{e\in Out_{v}} \mu(e)\CE(\beta(e))$
is compact (after proper enumeration of children of $v$).
    \item[(N3)]
The binary sum $\CE(\CP) = \sum_{v\in M} \nu(v) \CE(v)$ is in the compact form.
\end{itemize}

Power circuits $\CP_1$ and $\CP_2$ are {\em isomorphic} if
there exists a graph isomorphism $\varphi:\CP_1 \rightarrow \CP_2$
mapping $M(\CP_1)$ bijectively onto $M(\CP_2)$ and
preserving the values of $\mu$, $\nu$, an $\gamma$.

\begin{theorem}
Two constant power circuits in the normal form are equivalent if and only if they are isomorphic.
\end{theorem}

\begin{proof}
``$\Leftarrow$'' Obvious.

``$\Rightarrow$''
For $v_1\in V(\CP_1)$ we define $\varphi(v_1)$ to be the vertex $v_2\in V(\CP_2)$
such that $\CE(v_1) = \CE(v_2)$. Below we prove that for every $v_1$ there exists
$v_2$ with that property. Uniqueness of $v_2$ follows from the fact that $\CP_2$
is reduced.

Since $\CP_1$ and $\CP_2$ are equivalent we have
    $$\sum_{v\in M(\CP_1)} \nu(v) \CE(v) = \CE(\CP_1) = \CE(\CP_2) = \sum_{v\in M(\CP_2)} \nu(v) \CE(v),$$
where $\CE(v) = 2^{\sum_{e\in Out_{v}} \mu(e)\CE(\beta(e))}$.
By (N3) the sums for $\CE(\CP_1)$ and $\CE(\CP_2)$ are compact and, hence,
by Lemma \ref{le:compact_sum} are essentially the same (up to a permutation of summands).
Therefore, $\varphi$ defined above gives a one to one correspondence
between $M(\CP_1)$ and $M(\CP_2)$.

Suppose that $v_1 \in V(\CP_1)$ and $v_2 \in V(\CP_2)$ satisfy $\CE(v_1) = \CE(v_2)$.
Then
    $$\sum_{e\in Out_{v_1}} \mu(e)\CE(\beta(e)) = \sum_{e\in Out_{v_2}} \mu(e)\CE(\beta(e))$$
and both sums are in compact form by (N2).
By Lemma \ref{le:compact_sum} these sums are essentially the same and there is
one to one correspondence of the summands.

Finally, since $\CP_1$ and $\CP_2$ are trimmed, every vertex is a descendant of a marked
vertex. Therefore, we can inductively extend the one to one correspondence $\varphi$
from the marked vertices to all vertices of $\CP_1$. It is easy to see that $\varphi$
is a required graph isomorphism preserving values of $\mu$, $\nu$, and $\gamma$.
\end{proof}

\section{Reduction process}
\label{se:reduction}

The main goal of this section is to prove the following theorem, which is the main technical result of the paper.

\begin{theorem} \label{th:reduction}
There is an algorithm that given a constant power circuit $\CP$ constructs an equivalent reduced power circuit $\CP'$ in time
$O(|V(\CP)|^3)$. Moreover, $|V(\CP')| \leq |V(\CP)| + 1$.
\end{theorem}

We accomplish this in a series of lemmas and propositions.  The algorithm itself is described as Algorithm \ref{al:reduce} below.

\subsection{Geometric order}

In this section we present an algorithm which transforms a circuit
$\CP$ into a reduced one. Property (R1) and (R3) can be easily
achieved using Algorithm \ref{al:remove_redundancies}
which produces a trimmed strongly equivalent standard circuit of smaller size. Our
main goal is to find an algorithm that produces equivalent circuit
satisfying property (R2).

We say that a sequence $\{v_1,\ldots,v_n\}$ of vertices of $\CP$
is {\em geometrically ordered} if for each edge $e=v_i \rightarrow
v_j \in E(\CP)$ we have $i>j$.

\begin{lemma}
For any circuit $\CP$ there exists a geometric order on $V(\CP)$.
\end{lemma}

\begin{proof}

Induction on the number of vertices. Clearly a geometric
ordering
exists for $\CP$ with $|V(\CP)|=1$. Assume it exists for any 
directed graph $\CP$ without loops such that $|V(\CP)| <
N$. Let $\CP$ be a graph on $N$ vertices. Then by Lemma
\ref{le:zero_exists} there is a zero vertex $z$ in $\CP$.
Let $\CP'$ be obtained from $\CP$ by removing $z$ and
$\{v_1,\ldots,v_{N-1}\}$ be a geometric order of its
vertices. Then clearly $\{z,v_1,\ldots,v_{N-1}\}$ is a
geometric order on $V(\CP)$.

\end{proof}

\begin{lemma} \label{le:order_zero_one}
Assume that $\{v_1,\ldots,v_n\}$ is a geometric order on $V(\CP)$.
If $|V(\CP)| \ge 1$ then $v_{1}$ is a zero in $\CP$. If $|V(\CP)|
\ge 2$ and $\CP$ has the unique zero then $\CE(v_2) = 1$.
\end{lemma}

\begin{proof}
Clearly $Out_{v_1} = \emptyset$. Otherwise, there is $v\in
Out_{v_1}$ and by definition of geometric order, $v_1$
cannot precede $v$ which gives a contradiction.

If $\CP$ has a unique zero and $|V(\CP)| \ge 2$ then
$Out_{v_2} \ne \emptyset$. The only edge in $Out_{v_2}$
must be $(v_2,v_1)$ (otherwise we get a contradiction with
geometric order). Hence $\CE(v_2) = 1$.

\end{proof}

\subsection{Equivalent vertices}

\label{se:equal_elts}

In this section we define an inductive step for reduction of power
circuits. Let $\CP$ be a power circuit. We say that vertices $v_i$
and $v_j$ are called {\em equivalent} if $\CE(v_i) = \CE(v_j)$.
Clearly, $\CP$ satisfies (R2) if and only if it does not contain equivalent
vertices.

Assume that $\CP$ is a circuit satisfying (R1) and (R3) and $v_i,
v_j$ is the only pair of distinct equivalent vertices in $\CP$.
{\bf All circuits in this section are of this type.} In this
section we show how one can double the value of $\CE(v_j)$ in
$\CP$ while keeping $\CE$-values of all other vertices and the
value $\CE(\CP)$ the same. Using that algorithm we show
later how one can obtain a reduced circuit $\CP'$ equivalent to
$\CP$. The next algorithm transforms the given circuit $\CP$ so
that the vertices $v_i$, $v_j$ are not reachable from each other
along directed paths.

\begin{algorithm}\ \ \label{al:non_reachable}
    \\{\sc Input.}
A circuit $\CP$ satisfying the properties of this section.
    \\{\sc Output.}
An equivalent circuit $\CP'$ satisfying the properties of this
section such that $v_i$ and $v_j$ are not reachable from each
other.
    \\{\sc Computations.}
\begin{enumerate}
    \item[A)]
If $v_i$ is reachable from $v_j$ then:
\begin{enumerate}
    \item[(1)]
Remove all edges leaving $v_j$.
    \item[(2)]
For each edge $v_i \stackrel{\mu}{\rightarrow} u$ add an edge $v_j
\stackrel{\mu}{\rightarrow} u$.
    \item[(3)]
Output the obtained circuit.
\end{enumerate}
    \item[B)]
If $v_j$ is reachable from $v_i$ then perform steps as in the case
A. for $v_i$ and output the result.
    \item[C)]
If neither of $v_i$, $v_j$ is reachable from the other then output
$\CP$.
\end{enumerate}
\end{algorithm}

Observe that Algorithm \ref{al:non_reachable} does not change the
vertex set of $\CP$. In the next lemma we prove that the output of
Algorithm \ref{al:non_reachable} possesses all the claimed
properties.

\begin{lemma}\label{le:remove_edge_equivertices}
Let $\CP$ be a circuit on vertices $\{v_1,\ldots,v_n\}$, $\CP'$
the output of Algorithm \ref{al:non_reachable}, and $V(\CP') =
\{v_1',\ldots,v_n'\}$ where $v_i' \in V(\CP')$ corresponds to $v_i
\in V(\CP)$. Let $v_i$ and $v_j$ be two distinct vertices with
$\CE(v_i) = \CE(v_j)$. Then
\begin{enumerate}
    \item[1)]
$\CE(\CP) = \CE(\CP')$.
    \item[2)]
$\CE(v_i) = \CE(v_i')$ for every $k=1,\ldots,n$.
    \item[3)]
Neither $v_i$ nor $v_j$ is reachable from the other through
directed edges in $\CP'$
\end{enumerate}
Moreover, it takes linear time $O(|\CP|)$ to construct $\CP'$.
\end{lemma}

\begin{proof}
Assume that $v_i$ is reachable from $v_j$ in $\CP$ along a
directed path. By assumption of the lemma we have
$$\CE(v_j) = 2^{\left(\sum_{e \in Out_{v_j}} \mu(e) \CE(\beta(e))\right)} =
2^{\left(\sum_{e \in Out_{v_i}} \mu(e) \CE(\beta(e))\right)} = \CE(v_i)$$
and, therefore, $\sum_{e \in Out_{v_i}} \mu(e) \CE(\beta(e)) =
\sum_{e \in Out_{v_j}} \mu(e) \CE(\beta(e))$. Thus, replacing
edges leaving $v_j$ with edges leaving $v_i$ does not change
$\CE(v_j)$. Furthermore, it is easy to show that $\CE$-values of
all other vertices do not change. Finally, since the sign function
does not change the obtained circuit is equivalent to the initial
one. Clearly, the described procedure produces $\CP'$ in linear
time.

The case when $v_j$ is reachable from $v_i$ in $\CP$ is similar.
The case when neither of vertices can be
reached from the other is trivial.
\end{proof}

(Recall that each $\CP$ satisfies properties in the
beginning of this section.) The next algorithm makes values
$\CE(v_i)$ and $\CE(v_j)$ different by doubling the value
of $\CE(v_j)$. $\CE$-values of all other vertices remain
the same. For convenience we use the following notation
throughout the rest of the paper. We denote by $v^{(n)}$ a
vertex such that $\CE(v) = n$ (if it exists). Recall that
for each vertex $v \in \CP$ the value $\CE(v)$ is a power
of two. Hence if $n$ is not a power of $2$ then a vertex
$v^{(n)}$ does not exist in $\CP$.

\begin{algorithm} \label{al:separation}({\em Double $\CE$-value of a vertex})
$\CP' = Double(\CP,v_i,v_j)$.
    \\{\sc Input.}
A circuit $\CP$ on vertices $\{v_1,\ldots,v_n\}$ with the
specified pair of vertices $v_i$, $v_j$ such that $\CE(v_i) =
\CE(v_j)$.
    \\{\sc Output.}
An equivalent circuit $\CP'$ on vertices $\{v_1',\ldots,v_n'\}$
(with, maybe, one additional vertex $d$) such that $\CE(v_j') =
2\CE(v_j)$ and $\CE(v_k') = \CE(v_k)$ for $k\ne j$.
    \\{\sc Computations.}
\begin{enumerate}
 \item[A)] Apply Algorithm \ref{al:non_reachable} to vertices
 $v_i$, $v_j$ in $\CP$.
 \item[B)] Double the value $\CE(v_j)$ as follows:
\begin{enumerate}

 \item[1)] Compute the maximal number $N$ such that for each
$0\le k<N$ there exists an edge $v_j \stackrel{1}{\rightarrow}
v^{(2^k)}$ in $\CP$.

\item[2)] If $v^{(2^N)}$ does not exist in $\CP$ then (see
Figure \ref{fi:adding_one_to_power3})

\begin{enumerate}
 \item[a)] add a new vertex $d$ into $\CP$;
 \item[b)] connect $d$ to vertices in $\{ v^{(2^0)},\ldots,v^{(2^{N-1})} \}$
in such a way that $\CE(d) = 2^N$ (by Lemma
\ref{le:unique_positive_bin_sum} it is possible);
 \item[c)] remove all the edges $v_j \stackrel{1}{\rightarrow}
v^{(2^k)}$ (for each $0\le k<N$);

 \item[d)] add an edge $v_j \stackrel{1}{\rightarrow} d$.
\end{enumerate}

    \item[3)]
The case when $v^{(2^N)}$ exists in $\CP$ and there is an edge
$v_j \stackrel{-1}{\rightarrow} v^{(2^N)}$ is impossible since by
assumption there are no superfluous edges in $\CP$ .

 \item[4)] If $v^{(2^N)}$ exists in $\CP$ and
there is no edge between $v_j$ and 
$v^{(2^N)}$ then:

\begin{enumerate}
 \item[a)] remove all the edges $v_j \stackrel{1}{\rightarrow}
v^{(2^k)}$ (for each $0\le k<N$);

 \item[b)] add the edge $v_j \stackrel{1}{\rightarrow} v^{(2^N)}$.
\end{enumerate}
\end{enumerate}

 \item[C)] (Update edges) For each $k \in \{1,\ldots,n\} \setminus
\{i,j\}$ do the following:

\begin{enumerate}
 \item[1)] If there exist edges $e_1 = v_k \stackrel{\pm
1}{\rightarrow} v_i$ and $e_2 = v_k \stackrel{\pm
1}{\rightarrow} v_j$ (labels are equal) then erase the edge
$e_1$.

 \item[2)] If there exist edges $e_1 = v_k \stackrel{\pm 1}{\rightarrow} v_i$ and $e_2 =
v_k \stackrel{\mp 1}{\rightarrow} v_j$ (labels are
opposite) then erase both edges.

\item[3)] If there exists exactly one of the edges $e_1 =
v_k \stackrel{l}{\rightarrow} v_i$ and $e_2 = v_k
\stackrel{l}{\rightarrow} v_j$ then erase it and add an
edge $e_1 = v_k \stackrel{l}{\rightarrow} v_i$ (with the
same label).

\end{enumerate}

 \item[D)] Update marks on $v_i$ and $v_j$:

\begin{enumerate}
\item[1)] If both $v_i$ and $v_j$ are marked and $\nu(v_i)
= \nu(v_j)$ in $\CP$ then unmark $v_i$.

\item[2)] If both $v_i$ and $v_j$ are marked and $\nu(v_i)
= - \nu(v_j)$ in $\CP$ then unmark both $v_i$ and $v_j$.

\item[3)] If exactly one of $v_i$, $v_j$ has a mark $\nu =
\pm 1$ in $\CP$ then unmark it and mark $v_i$ with $\nu$.
\end{enumerate}

 \item[E)] Output the result.
\end{enumerate}

\end{algorithm}

    \begin{figure}[htbp]
    \centerline{
    \includegraphics[scale=0.6]{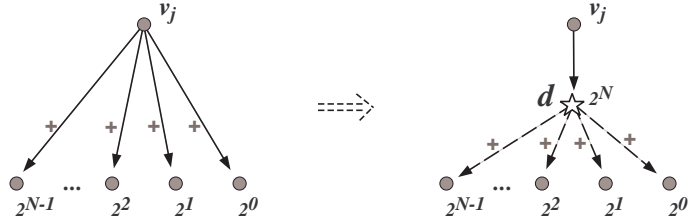} }
    \caption{\label{fi:adding_one_to_power3} Introducing auxiliary vertex $d$
    (case B.2 of Algorithm \ref{al:separation}). The dashed lines denote ``possible'' edges.
    Grey vertices can be marked or unmarked.}
    \end{figure}

Let $\CP'$ be the result of an application of Algorithm
\ref{al:separation} to $\CP$. Observe that Algorithm
\ref{al:separation} does not remove any vertices, but might
introduce a new vertex into $\CP$ at step B.2. We will refer to
this vertex as an {\em auxiliary} vertex and denote it by $d$.
Furthermore, if $V(\CP) = \{v_1,\ldots,v_n\}$ then $V(\CP')$
contains vertices $\{ v'_1, \ldots, v'_n \}$ where each $v_k' \in
V(\CP')$ corresponds to $v_k \in V(\CP)$, and perhaps a new vertex
$d$ which we call an {\em auxiliary vertex}.

\begin{proposition} \label{pr:double_equivalence}
Let $\CP' = Double(\CP,v_i,v_j)$. Then $\CP'$ is equivalent to
$\CP$. Moreover, $\CE(v_j') = 2 \CE(v_j)$ and $\CE(v_k') =
\CE(v_k)$ for each $k \ne j$.
\end{proposition}

\begin{proof}
By definition $\CE(v_j) = 2^{p_j}$, where $p_j = \sum_{e \in
Out_{v_j}} \mu(e) \CE(\beta(e))$. Let $N$ be the maximal number
such that for each $k<N$, there exists an edge $v_j
\stackrel{1}{\rightarrow}v^{(2^k)}$ in $\CP$ (computed at step
B.1). Steps B.2), B.3), and B.4) are mutually exclusive. We
consider only one case defined in B.2 ($\CP$ does not contain a
vertex $v^{(2^N)}$) Other cases can be considered similarly.

In the case under consideration the algorithm creates an unmarked
vertex $d$ such that $\CE(d) = 2^N$, removes edges $\{v_j
\stackrel{1}{\rightarrow} v^{(2^0)}, \ldots, v_j
\stackrel{1}{\rightarrow} v^{(2^{N-1})}$ leaving $v_j$, and adds
an edge $v_j \stackrel{1}{\rightarrow} d$. Clearly, after these
transformations $\CE(v_j') = 2^{p_j'}$ where
$$p_j' = p_j - ( 2^{N-1} + \ldots + 2^{0} ) +2^N = p_j+1$$
and, therefore,  after the step B.2 we have $\CE(v_j') = 2
\CE(v_j)$.

Let $v_k$ be a vertex in $\CP$ such that there is an edge $v_k
\stackrel{\pm 1}{\rightarrow} v_j$. Since $\CE(v_j)$ has been
changed the value $\CE(v_k)$ has been changed too. On step C) the
algorithm recovers $\CE$-values of these vertices. It is
straightforward to check (using the definition of $\CE$) that
after the step C) we have $\CE(v_k') = \CE(v_k)$ for all $k\ne j$.

Finally, at step D), Algorithm \ref{al:separation} makes sure that
$\CE(\CP) = \CE(\CP')$. If the vertex $v_j$ is marked then after
the step D) we have  $\CE(\CP') = \CE(\CP) + \nu(v_j) \CE(v_j)$.
To get the equality back Algorithm \ref{al:separation} performs
the step D), which removes the additional summand.

\end{proof}

\begin{proposition}
The time-complexity of Algorithm \ref{al:separation} is
$O(n)$ where $n$ is the number of vertices in $\CP$.
\end{proposition}

\begin{proof}
Performing step A) requires $O(n)$ operations, step B) requires
$O(n)$ operations, step C) requires $O(n)$ operations, step D)
requires $O(1)$ operations.

\end{proof}

Observe, that after doubling the value $\CE(v_j)$ we have
$\CE(v_j) = 2\CE(v_i) \ne \CE(v_i)$. But it is possible that there
is a vertex $v_k' \in \CP'$ such that $\CE(v_j') = \CE(v_k')$. In
this case we have to double the value $\CE(v_j')$ again. We
formalize it in the following algorithm.

\begin{algorithm} \label{al:separation_total}({\em $\CE$-value separation})
$\CP' = Separate(\CP,v_j)$.
    \\{\sc Input.}
A circuit $\CP$ with equivalent vertices $v_i$ and $v_j$.
    \\{\sc Output.}
A reduced circuit $\CP'$ equivalent to $\CP$.
    \\{\sc Computations.}
\begin{enumerate}

\item[A)] Double the value $\CE(v_j)$ in $\CP$ by Algorithm
\ref{al:separation}. Denote the result by $\CP'$ (i.e.
$\CP' \leftarrow Double(\CP,v_j)$).

\item[B)] Trim $\CP'$ (using Algorithm \ref{al:trim}).
Denote the result by $\CP$ (i.e. $\CP \leftarrow
Trim(\CP')$).

\item[C)] If $v_j$ was not removed on step B) (see Figure
\ref{fi:reduction_decrease} for example) and there exists a vertex
$v_k \in \CP$ such that $\CE(v_k) = \CE(v_j)$ then goto A) and
repeat for $v_k$ and $v_j$.

\item[D)] Output $\CP$.

\end{enumerate}
\end{algorithm}

\begin{figure}[htbp]
\centerline{
\includegraphics[scale=0.5]{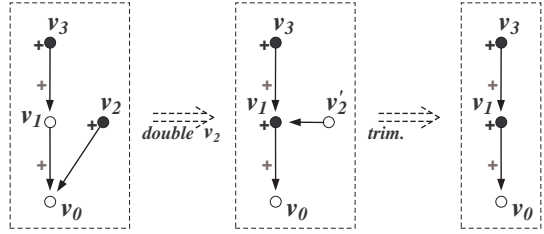}}
\caption{\label{fi:reduction_decrease} Doubling of $\CE(v_2)$
results in a non-trimmed circuit and a removal of $v_2'$. Only one
iteration (application of Algorithm \ref{al:separation}) is
performed.}
\end{figure}

\begin{proposition}\label{pr:separation_equivalence}
Let $\CP'$ be the output of Algorithm \ref{al:separation_total} on
a circuit $\CP$. Then $\CP'$ is equivalent to $\CP$.
\end{proposition}

\begin{proof}
Follows from Proposition \ref{pr:double_equivalence}.
\end{proof}

One can find an upper bound for the number of iterations Algorithm
\ref{al:separation_total} performs to make $\CE(v_j)$ unique among
$\CE$-values of vertices in $\CP$. Let
\begin{equation} \label{eq:sep_seq}
v_{a_1},\ldots,v_{a_s}
\end{equation}
be a sequence of vertices in $\CP$ such that:
\begin{enumerate}
    \item[1)]
$\CE(v_{a_1}) = \CE(v_j)$, where $v_{a_1}$ and $v_j$ are distinct
vertices;
    \item[2)]
$\CE(v_{a_{k+1}}) = 2 \CE(v_{a_{k}})$ for each $k=1,\ldots,s-1$;
    \item[3)]
$s$ is the maximal length of a sequence with such properties.
\end{enumerate}
We call the sequence (\ref{eq:sep_seq}) satisfying all the
properties above the {\em separation sequence} for $v_j$ in $\CP$.
The number of iterations Algorithm \ref{al:separation_total}
performs to separate $v_j$ is not greater than $s$ as shown in
Figure \ref{fi:reduction_decrease}.

As we mentioned earlier each application of Algorithm
\ref{al:separation} can introduce one auxiliary vertex. Algorithm
\ref{al:separation_total} invokes Algorithm \ref{al:separation}
several times and, hence, several auxiliary vertices might be
introduced. In the next proposition we show that an application of
Algorithm \ref{al:separation_total} introduces at most one
additional vertex.

\begin{proposition} \label{pr:new_down_vertex}
Let $\CP$ be a power circuit. Suppose $v_i$ and $v_j$ are the only
two equivalent vertices in $\CP$. Let $v_{a_1}, \ldots, v_{a_s}$
be a separation sequence for $v_j$. Then during the separation of
$\CE(v_j)$ from $\CE(v_{a_1}), \ldots, \CE(v_{a_s})$ an auxiliary
vertex can be introduced at the last iteration only. Therefore,
Algorithm \ref{al:separation_total} can introduce at most one
auxiliary vertex and if $\CP' = Separate(\CP,v_i,v_j)$ then
$|V(\CP')| \le |V(\CP)|+1$.
\end{proposition}

\begin{proof}
Assume that an auxiliary vertex $d$ was introduced on $k$th
iteration $(k<s)$ (when separating $\CE(v_j)$ and $\CE(v_{a_k})$).
Let $\CE(d) = 2^N$ for some $N \in \mathbb{N}$. Denote by $v_j'$
the vertex $v_j$ after the $k$th iteration (technically $v_j$ and
$v_j'$ belong to different graphs). See Figure
\ref{fi:seq_doubling}.
\begin{figure}[htbp]
\centerline{
\includegraphics[scale=0.5]{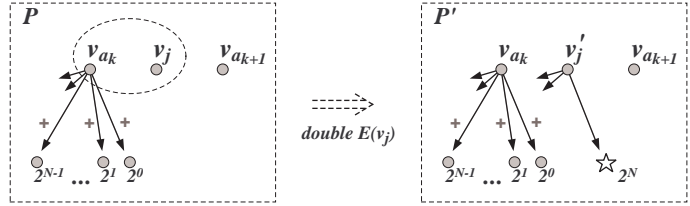}}
\caption{\label{fi:seq_doubling} A situation when an auxiliary
vertex $d$ (such that $\CE(d) = 2^N$), denoted by the star, was
introduced in the middle of separation. In this case we argue that
the vertex $v_{a_{k+1}}$ must be already connected to a vertex
with $\CE$-value equal to $2^N$.}
\end{figure}

Consider vertices $v_{a_k}$ and $v_{a_{k+1}}$ in the circuit
before the $k$th iteration. We have $\CE(v_{a_k}) = \CE(v_j) =
2^{p}$ and $\CE(v_{a_{k+1}}) = \CE(v_j') = 2^{p'}$, where
$$p = \sum_{e \in Out_{v_j}} \mu(e) \CE(\beta(e)) = \ldots + (2^{N-1} + \ldots + 2^0)$$
and
    $$p' = \sum_{e \in Out_{v_j'}} \mu(e) \CE(\beta(e)) = p+1 = \ldots + 2^{N}.$$
Observe that since vertices $v_j,v_{a_k}$ are not reachable from
each other and since $v_j,v_{a_k}$ is the only pair of equivalent
vertices it follows that the binary sum above for $p$ is reduced.
Also, by our assumption (auxiliary vertex is created after the
$k$th separation), there is no edge $e \in Out_{v_{a_k}}$ such
that $\CE(\beta(e)) = 2^N$. Therefore, $2^N$ is the smallest
summand in the binary sum for $p'$ above and $p'$ is divisible by
$2^N$. Hence, $v_{a_{k+1}}$ cannot be connected to vertices
$v^{(2^0)}, \ldots, v^{(2^{N-1})}$ (otherwise it would contradict
divisibility of $p'$ by $2^N$ or the fact that $\CP$ does not
contain multiple edges). And there must exist a vertex $v^{(2^N)}$
in the circuit before the $k$th separation and $v_{a_{k+1}}$ must
be connected to it (follows from Lemma
\ref{le:binary_divisibility}). Obtained contradiction finishes the
proof.

\end{proof}

For some circuits it is impossible to avoid adding a new vertex
when performing separation. Figure \ref{fi:new_vertex} illustrates
the case.
\begin{figure}[htbp]
\centerline{
\includegraphics[scale=0.5]{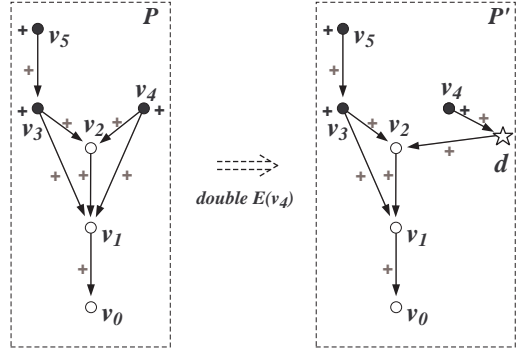} }
\caption{\label{fi:new_vertex} A situation when it is necessary to
introduce an auxiliary vertex to ``separate'' vertices $v_3$ and
$v_4$ ($\CE(v_3) = \CE(v_4) = 8$). To double the value $\CE(v_4)$
we have to add one to the power of $v_4$ which is
$\CE(v_1)+\CE(v_2) = 2+1$ initially. Therefore, a new vertex $d$
such that $\CE(d) = 4$ is required.}
\end{figure}

\begin{proposition}
The time complexity of Algorithm \ref{al:separation_total}
is $O(n^2)$, where $n$ is the number of vertices in $\CP$.
\end{proposition}

\begin{proof}
In the worst case one has to double the value of $v_j$ at most $n$
times which makes the complexity of Algorithm
\ref{al:separation_total} at most quadratic in terms of
$|V(\CP)|$.

\end{proof}

Finally, notice that in Algorithm \ref{al:separation} we assume
that we know the $\CE$-values of vertices. In the next section we
explain how it can be achieved.

\subsection{Reduction process}

\label{se:reduction_process}

In this section we present an algorithm which transforms any power
circuit into an equivalent reduced one. Moreover, we show that
this operation can be performed in polynomial time in terms of the
size of the input.

First, we describe the idea of the algorithm. Let $\CP$ be a
trimmed circuit without redundant zeros (described in Section
\ref{se:standard-circuits}). Assume that, in addition, we are provided with a
subset $C$ of $V(\CP)$ satisfying the following properties:
\begin{enumerate}
 \item[(C1)] For $u,v \in C$, $\CE(u) = \CE(v)$ if and only if $u =
v$ (i.e., property (R2) holds inside $C$).

 \item[(C2)] If $u \in C$ and $u \rightarrow v$ is an edge leaving $u$ then
$v \in C$ ($C$ is itself a circuit).
\end{enumerate}
Moreover, assume that we have the following additional
information about $C$:
\begin{enumerate}
\item[1)] Vertices from $C$ are ordered with respect to
their $\CE$-values. In other words there exists a sequence
$c_1,\ldots,c_m$ such that $C = \{v_{c_1},\ldots,v_{c_m}\}$
and $\CE(v_{c_i}) < \CE(v_{s_j})$ if and only if $i<j$.

\item[2)] There is a sequence $d_{1},\ldots,d_{m-1}$ of
$0$'s and $1$'s such that $d_i = 1$ if and only if
$\CE(v_{c_{i+1}}) = 2 \CE(v_{c_{i}})$.
\end{enumerate}
An iteration of the reduction process transforms $\CP = (V,E)$ and
extends the set $C$ so that the number $|V|-|C|$ decreases by at
least one. The reduction procedure works until $C = V$. The main
ingredient is Algorithm \ref{al:separation_total} described in
Section \ref{se:equal_elts}.

The initial set $C$ is computed as follows. Let $\{v_{1},\ldots,
v_{n}\}$ be a geometric order on $V(\CP)$. If $n = 1$ then
$\CE(\CP) = 0$ and $\CP$ is reduced. Suppose $n \ge 2$. It follows
from Lemma \ref{le:order_zero_one} that $\CE(v_{1}) = 0$ and
$\CE(v_{2}) = 1$. Put $C= \{v_1,v_2\}$, $c_1=1$, $c_2 = 2$, and
$d_1 = 0$. This is the basis of computations.

Now we describe one iteration. If $m = n$ then $\CP$ is reduced
and there is nothing to do. Suppose that $m \ne n$. Since $\CP$
has no loops, there exists a vertex $v\in V \setminus C$ such that
$C' = C \cup \{v\}$ satisfies property (C2). Our main goal is to
make $C'$ satisfy property (C1). In the next lemma we show some
computational properties of the set $C'$.

\begin{lemma}\label{le:compare_vertices}
Let $v$ be a vertex in $V \setminus C$ such that $C \cup \{v\}$
satisfies property (C2). For any vertex $u \in C$ one can compare
values $\CE(v)$ and $\CE(u)$ and check if $\CE(v) = 2 \CE(u)$ or
$\CE(u) = 2 \CE(v)$. Furthermore, the time complexity of this
operation is $O(|C|)$.
\end{lemma}

\begin{proof}

By definition $\CE(v) = 2^{p_v}$ and $\CE(u) = 2^{p_u}$ where
$$p_v = \sum_{e \in Out_v} \mu(e) \CE(\beta(e)) \mbox{ \ and \ } p_u = \sum_{e \in Out_u} \mu(e) \CE(\beta(e)).$$
Clearly, $\CE(v) < \CE(u)$ if and only if $p_v < p_u$. Hence, it is
sufficient to compare $p_v$ and $p_u$. It follows from the choice of
$v$ and property (C2) that edges leaving $v$ an $u$ have termini in
$C$ and, therefore, the binary sums above for $p_v$ and $p_u$ are
reduced by (C1). Moreover, by assumption, vertices from $C$ are
ordered with respect to their $\CE$-values and we know $\CE$-values
of which of them are doubles $\CE$-values of others vertices
(provided by the sequence $d_1,\ldots,d_{m-1}$). This information is
clearly enough to use Algorithm \ref{al:compare_binaries} which has
linear time complexity by Proposition \ref{pr:complexity_compare}.
Since $|Out_u| \le |C|$ and $|Out_v| \le |C|$ the linearity of the
process follows.

Finally, since Algorithm \ref{al:compare_binaries} can determine
if $p_u$ and $p_v$ differ by $\pm 1$, one can determine whether
$\CE(v) = 2 \CE(u)$ or $\CE(u) = 2 \CE(v)$.

\end{proof}

Now we can describe the inductive step. By Lemma
\ref{le:compare_vertices} one can compare the vertex $v$ with any
vertex $u \in C$ and, hence, find a position of $v$ in the ordered
sequence $\{ v_{c_1},\ldots,v_{c_m}\} = C$. (Observe that to find
a position of $v$ one does not have to compare $v$ with each $u
\in C$. Instead, this can be achieved by a binary search in at
most $\log_2 m$ comparisons.) There are two outcomes of the
comparison of $v$ with the vertices from $C$ possible. First, if
for each $u \in C$ $\CE(v) \ne \CE(u)$ then we can add $v$ into
$C$ without any modification of a current circuit and update the
sequences $\{c_1,\ldots,c_m\}$ and $\{d_1,\ldots,d_{m-1}\}$
according to the results of comparison. After that $V \setminus C$
becomes smaller and induction hypothesis applies.

In the second case there exists a vertex $u \in C$ such that
$\CE(v) = \CE(u)$. In this case we apply Algorithm
\ref{al:separation_total} to $v$ to make $\CE(v)$ different from
values $\{\CE(v_{c_1}),\ldots,\CE(v_{c_m})\}$. We would like to
emphasize here that the new value $\CE(v)$ might be equal to
$\CE(w)$ for some $w \in V \setminus C$, but it is unique in $C
\cup \{v\}$. After that we can add $v$ into $C$ and update the
order. Also, notice that after the separation an auxiliary vertex
might appear. But since it has a unique $\CE$-value (in $C$) we
can add it into $C$ too. It follows that $|V \setminus C|$ becomes
smaller and induction hypothesis applies.

\begin{algorithm} \label{al:reduce}{\em (Reduction)} $\CP' = Reduce(\CP)$.\\
{\sc Input.} A circuit $\CP$ \\
{\sc Output.} A reduced circuit $\CP'$ equivalent to $\CP$.\\
{\sc Initialization.} $C = \emptyset$.\\
{\sc Computations.}
\begin{enumerate}
 \item[A)] Let $\CP_1 = Trim(\CP)$.
 \item[B)] $\CP_2 = RemoveRedundancies(\CP_1)$ (Algorithm \ref{al:remove_redundancies}).
 \item[C)] Order vertices $V(\CP_2)$ with respect to the geometry of $\CP_2$
 $$V(\CP_2) = \{v_1,\ldots,v_n\}.$$
 \item[D)] Put $C = \{v_1,v_2\}$ and, accordingly, initialize sequences $c_1,c_2$ and $d_1$.
 \item[E)] For each vertex $v \in \{ v_3,\ldots,v_n \}$ (in the order defined by indices) perform the following operations:
\begin{enumerate}
 \item[1)] remove opposite and superfluous pairs of edges leaving $v$;
 \item[2)] using binary search and Algorithm
 \ref{al:compare_binaries} find a position of $v$ in $C$;
 \item[3)] if necessary separate the vertex $v$ from vertices from
 $C$ using Algorithm \ref{al:separation_total};
 \item[4)] add $v$ and, perhaps, a new auxiliary vertex $d$ into $C$ and update the order on $C$.
\end{enumerate}
 \item[F)] Output the obtained circuit.
\end{enumerate}
\end{algorithm}

The sequence of operations E.1)-E.4) applied to $v \in
\{v_3,\ldots,v_n\}$ will be referred to as processing of the
vertex $v$. Figure \ref{fi:reduction_example} illustrates the
execution of Algorithm \ref{al:reduce} for a particular circuit.

\begin{proposition} \label{pr:reduction_result}
Let $\CP' = Reduce(\CP)$. Then $\CE(\CP) = \CE(\CP')$.
\end{proposition}

\begin{proof}
Follows from Propositions
\ref{pr:red_zeros_equivalence}, and
\ref{pr:separation_equivalence}.
\end{proof}

\begin{figure}[htbp]
\centerline{
\includegraphics[scale=0.5]{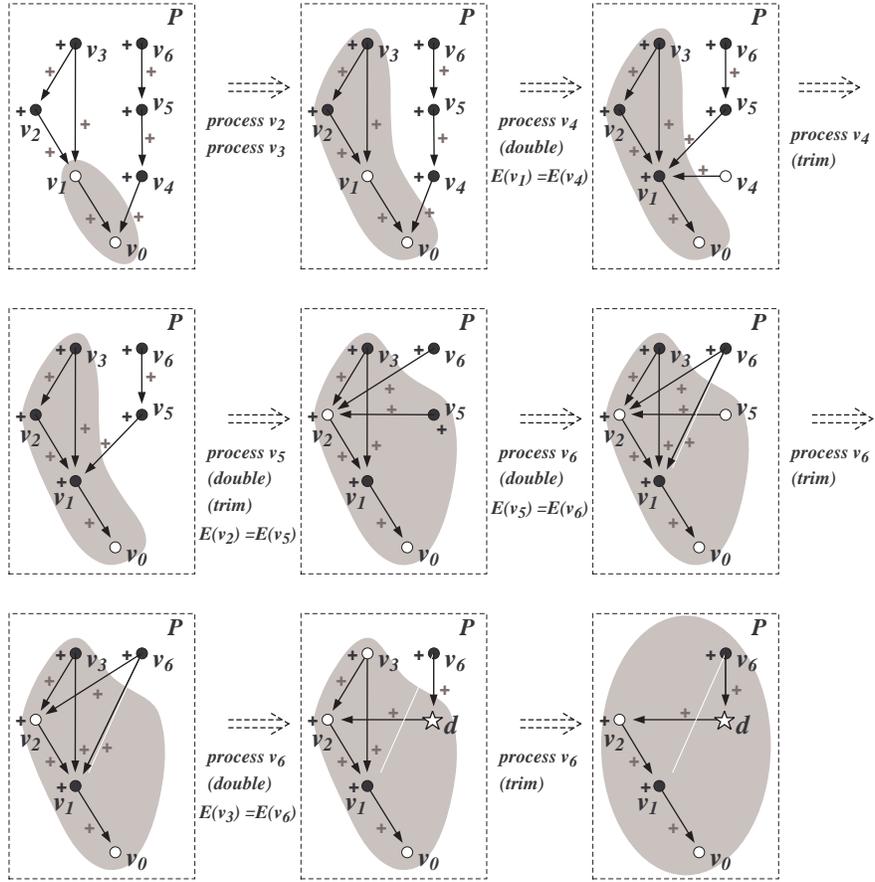} }
\caption{\label{fi:reduction_example} Reduction of a circuit.
Initially $V(\CP) = \{v_0,\ldots,v_6\}$, where vertices are
ordered in geometric order. Grey regions encompass vertices
belonging to $C$.}
\end{figure}

By Proposition \ref{pr:new_down_vertex} each separation might
introduce at most one new auxiliary vertex. Therefore, in the
worst case an application of Algorithm \ref{al:reduce} to $\CP$
can introduce $n-2$ new vertices. In the next proposition we show
that the number of vertices in $\CP$ after an application of
Algorithm \ref{al:reduce} can increase by at most $1$.

\begin{proposition} \label{pr:reduction_and_vertices}
Let $\CP' = Reduce(\CP)$. Then $|V(\CP')| \le |V(\CP)|+1$.
\end{proposition}

\begin{proof}
It is convenient to introduce the following notation. Let $v$ and
$v'$ be two vertices in $\CP$ such that $\CE(v) = 2 \CE(v')$. In
this event we say that $v'$ is a {\em half} of $v$ and denote it
by $\widetilde{v}$. Also, denote by $\CP_k$ a circuit obtained
after processing a vertex $v_k$ (where $3\le k \le n$) and by
$C_k$ the set of checked vertices in $\CP_k$. For notational
convenience define $\CP_2 = \CP$ and $C_2 = \{v_1,v_2\}$.
Schematically,
$$
 (\CP,\{v_1,v_2\}) = (\CP_2,C_2)
 \stackrel{process ~v_3}{\longrightarrow}
 (\CP_3,C_3)
 \stackrel{process ~v_4}{\longrightarrow}
 \ldots
 \stackrel{process ~v_n}{\longrightarrow}
 (\CP_n,C_n).
$$

The number of vertices in $\CP_k$ changes at steps E.3 only when
Algorithm \ref{al:separation_total} is used. Recall that Algorithm
\ref{al:separation_total}:
\begin{enumerate}
    \item[$\bullet$]
can introduce at most one auxiliary vertex;
    \item[$\bullet$]
remove some of the vertices while trimming the result (Algorithm
\ref{al:separation_total} step B).
\end{enumerate}
By Proposition \ref{pr:new_down_vertex} $|V(\CP_{k+1})| -
|V(\CP_{k})| \le 1$ and $|V(\CP_{k+1})| - |V(\CP_{k})| = 1$ if and
only if processing of $v_{k+1}$ introduced a new auxiliary vertex
and no other vertices were removed while trimming. Therefore, to
prove the statement of the proposition it is sufficient to prove
the following assertion.

\vspace{3mm}\noindent{\bf Main assertion.} Let $s,t$ be two
integers such that $3\le s<t\le n$,
$$|V(\CP_{s-1})|+1 = |V(\CP_s)| = \ldots = |V(\CP_{t-1})|,$$
and there were no vertices removed and no auxiliary vertices
introduced while processing $v_{s+1},\ldots,v_{t-1}$. Processing
of $v_t$ cannot introduce an auxiliary vertex.

 \vspace{3mm}
Let $v_{a_1},\ldots,v_{a_k} \in V(\CP_{s-1})$ be a separation
sequence for $v_s$. It follows from Proposition
\ref{pr:new_down_vertex} that there are edges $v_{a_{k}}
\stackrel{1}{\rightarrow} v^{(2^0)},\ldots,v_{a_{k}}
\stackrel{1}{\rightarrow} v^{(2^{N-1})}$ in $\CP_{s-1}$ (just
before the separation of $v_s$). After the processing of $v_s$
there is an edge $v_s \stackrel{1}{\rightarrow} d$ and no edges
$v_s \stackrel{\pm 1}{\rightarrow} v^{(2^0)}, \ldots, v_s
\stackrel{\pm 1}{\rightarrow} v^{(2^{N-1})}$ in $\CP_{s}$. Recall
that the auxiliary vertex $d$ is created unmarked and there is
only one edge incoming into $d$ which is $v_s
\stackrel{1}{\rightarrow} d$. Moreover, the following claim is
true.

\vspace{3mm}\noindent{\bf Claim 1.} With our assumptions on
$v_s,\ldots,v_{t-1}$ the following is true for each $\CP_k$ ($s\le
k\le t-1$):
\begin{enumerate}
    \item[(D1)]
For each $m = 0,\ldots,N$ there is a vertex $v \in C_k$ such that
$\CE(v) = 2^m$.
    \item[(D2)]
The vertex $d$ is unmarked in $C_k$.
    \item[(D3)]
If there is an edge $w \stackrel{\pm 1}{\rightarrow} d$ in $\CP_k$
then $w \in C_k$ (i.e., only vertices from $C_k$ can be connected
to $d$).

Furthermore, if the edge $e = w \stackrel{\pm 1}{\rightarrow} d$
does not exist in $\CP_{k'}$ then it does not exist in any
$P_{k}$, where $k'<k\le t-1$.
    \item[(D4)]
Let $w$ be a vertex connected to $d$. If $\CE(w) = 2^{p_w}$ then
$2^N$ is the smallest summand in the corresponding binary sum $p_w
= \sum_{e \in Out_w} \mu(e) \CE(\beta(e))$.
    \item[(D5)]
The vertex $v_{k+1} \in \CP_k$ is not equivalent to any vertex
connected to $d$.
    \item[(D6)]
If a vertex $v$ is connected to $d$ then $\widetilde{v}$ is
present in $\CP$ and at least one of $v$, $\widetilde{v}$ is
unmarked.
    \item[(D7)]
For any vertex $w \in \CP$ and a vertex $v$ connected to $d$ there
exists at most one edge $w \stackrel{\pm 1}{\rightarrow} v$ or $w
\stackrel{\pm 1}{\rightarrow} \widetilde{v}$.
\end{enumerate}

\begin{proof}
By induction on $k$. Suppose $k = s$. Properties (D1)-(D3) are
already proved in the remark preceding the claim. Since $v_s$ is
connected to $d$ and is not connected to $v^{(2^0)}, \ldots,
v^{(2^{N-1})}$ the property (D4) is established. To show (D5)
consider the vertex $v_{s+1} \in \CP_{s} \setminus C_s$ and prove
that $\CE(v_s) \ne \CE(v_{s+1})$. Since all vertices leaving
$v_{s+1}$ have termini in $C_{s}$ and $C_s$ is a reduced part of
$\CP_s$ it follows that $\CE(v_{s+1}) = 2^{p_{v_{s+1}}}$ where
$p_{v_{s+1}} = \sum_{e \in Out_{v_{s+1}}} \mu(e) \CE(\beta(e))$ is
a reduced binary sum which does not involve $2^N$ (by D3). As we
showed above $p_{v_s}$ is a reduced binary sum which contains
$2^N$. Therefore, (D5) follows from Lemma
\ref{le:binary_divisibility}. Properties (D6) and (D7) follow from
the description of Algorithm \ref{al:separation}.

Assume that (D1)-(D7) hold for each $k$ such that  $s \leq k < K
\leq t-2$ and show that they hold for $k = K$.

\vspace{3mm}\noindent{\bf (D1)} By induction assumption
vertices $v^{(2^0)},\ldots,v^{(2^{N-1})},v^{(2^N)}$ are
present in $\CP_{K-1}$. Since no vertices are removed while
processing $v_K$ the property (D1) holds for $K$.

\vspace{3mm}\noindent{\bf (D2)} Let $\{ v_{a_1}, \ldots, v_{a_m}\}
\subseteq C_{K-1}$ be a separation sequence for $v_K$. By
induction assumption the vertex $d$ is unmarked in $C_{K-1}$.
Assume, to the contrary, that $d$ is marked in $C_K$. Then $d$
must belong to $\{v_{a_1}, \ldots, v_{a_m}\}$ and, hence,
$\CE(v_K) \le 2^N$. We claim that in this case processing of $v_K$
results in a removal of $v_K$ which will contradict to the
assumption of the claim (no vertices removed).

Indeed, Algorithm \ref{al:separation_total} consequently doubles
$\CE(v_K)$ (using Algorithm \ref{al:separation}) and trims
intermediate results. Consider a step when $\CE(v_K) = 2^N$ and we
double $\CE(v_K)$ to separate it from $\CE(d) = 2^N$. Denote the
circuit before that separation by $\CP_K'$ and after it by
$\CP_K''$. The vertex $d$ is unmarked in $\CP_K'$ and marked in
$\CP_K''$. Therefore, the vertex $v_K$ is unmarked in $\CP_K''$
since $d$ is unmarked in $\CP_K'$ (follows from the description of
$Double$-procedure).

Furthermore, we claim that $v_K$ has no incoming edges in
$\CP_k''$. Indeed, consider two cases. Let $w \in C_{K-1}$. Then,
initially, there is no edge $w \rightarrow v_K$ in $\CP_{K-1}$
(guaranteed by property (C2) for $C_{K-1}$) and, therefore, when
we continuously double the value $\CE(v_K)$ there is no need to
introduce $w \rightarrow v_K$. Assume $w \not \in C_{K-1}$. Then
there is no edge $w \rightarrow d$ by (D3) for $C_{K-1}$.
Therefore, even if the edge $w \rightarrow v_K$ would existed, it
would be removed in $\CP_K''$ (when separating $v_K$ from $d$).

Thus, since $v_K$ is not marked and has no incoming edges in
$\CP_K''$ it will be removed while trimming $\CP_K''$. This
contradicts to the assumption that no vertices are removed.

\vspace{3mm}\noindent{\bf (D3)} There are three cases how a vertex
$w$ can become connected to $d$ while processing $v_K$:
\begin{enumerate}
    \item[1)]
$d$ belongs to the separation sequence of $v_K$ (and $w$ is
connected to $v_K$);
    \item[2)]
$w = v_K$ and a vertex $v$ connected to $d$ belongs to the
separation sequence of $v_K$;
    \item[3)]
$w = v_K$ and a vertex $v$ for which there are edges $v
\stackrel{1}{\rightarrow} v^{(2^0)}, \ldots, v
\stackrel{1}{\rightarrow} v^{(2^{N-1})}$ and no edge $v
\stackrel{1}{\rightarrow} v^{(2^{N})}$ belongs to the separation
sequence of $v_K$.
\end{enumerate}
The first case, as shown in (D2), raises a contradiction.
Therefore, only the vertex $v_K$ can become connected to $d$.
Since it is being added to $C_K$ it does not contradict to (D3).

Furthermore, if a vertex $v_K$ is not connected to $d$ in $\CP_K$
it is not connected to $d$ in each $\CP_k$ ($K\le k\le t-1$).

\vspace{3mm}\noindent{\bf (D4)} As shown in (D3) after the
processing of the vertex $v_K$ the set $In_d$ can increase by at
most one element $v_K \stackrel{\varepsilon}{\rightarrow} d$
(i.e., processing of $v_K$ can connect to $d$ only the vertex
$v_K$). Assume that $v_K$ is connected to $d$ in $\CP_{K+1}$ and
contradicts to (D4). This might happen only in the second case in
the proof of (D3), i.e., some vertex $v$ connected to $d$ belongs
to the separation sequence of $v_K$. Let $v_{a_b}$ be the first
vertex in the separation sequence $\{ v_{a_1}, \ldots, v_{a_m}\}
\subseteq C_{K-1}$ of $v_K$ connected to $d$. We argue (as in the
proof of (D2)) that separation of $v_K$ results in a removal of
$v_K$ from the circuit.

The vertex $v_K$ is not connected in $\CP_{K-1}$ to $d$ by (D3).
Moreover, $v_K$ is not equivalent to any vertex in $\CP_{K-1}$
connected to $d$ by (D5). Therefore, $v_{a_b}$ is not the first
vertex in the separation sequence of $v_K$, it must be preceded by
$\widetilde{v}_{a_b}$ (which by (D6) exists in $C_{K-1}$).
Consider a step of doubling of $\CE(v_K)$ when $\CE(v_K) =
\CE(v_{a_b})$. Denote by $\CP_K'$ the circuit before that step and
by $\CP_K''$ the result of doubling.

The vertex $v_K$ in $\CP_K''$ is unmarked since by (D6) either
$\widetilde{v}_{a_b}$ or $v_{a_b}$ is unmarked in $\CP_{K-1}$.
Also, by (D7) for each $w\in \CP_{K-1}$ there is at most one edge
$w \stackrel{\pm 1}{\rightarrow} \widetilde{v}_{a_b}$ or $w
\stackrel{\pm 1}{\rightarrow} v_{a_b}$. Therefore, in $\CP_K''$
$v_K$ has no incoming edges. Thus, $v_K$ will be removed while
trimming $\CP_K'$. This contradicts to our assumption that no
vertices are removed.

The obtained contradiction implies that a vertex connected to $d$
cannot belong to the separation sequence of $v_K$. Therefore, if
the vertex $v_K$ is connected to $d$ in $\CP_K$ then it cannot be
connected to vertices with smaller $\CE$-values ($2^0, \ldots,
2^{N-1}$) and, hence, $\CE(d) = 2^N$ is the least summand in the
power of $\CE(v_K)$.

\vspace{3mm}\noindent{\bf (D5)} Let $v$ be a vertex connected to
$d$ in $\CP_K$. By (D4) $2^N$ is the least summand in the power of
$\CE(v)$. Hence, if $v_{K+1}$ is equivalent to $v$ then by Lemma
\ref{le:binary_divisibility} $v_{K+1}$ must be connected to a
vertex with $\CE$-value $2^N$. But termini of the edges leaving
$v_{K+1}$ belong to $C_K$. Therefore, $v_{K+1}$ must be connected
to $d$ since it is the only vertex in $C_K$ with $\CE$-value
$2^N$. Contradiction to (D3).

\vspace{3mm}\noindent{\bf (D6)} As shown in the proof of (D3) and
(D4) a vertex can become connected to $d$ only when its separation
sequence contains a vertex $v_{a_b}$ for which there are vertices
$v_{a_b} \stackrel{1}{\rightarrow} v^{(2^0)}, \ldots, v_{a_b}
\stackrel{1}{\rightarrow} v^{(2^{N-1})}$ and there is no edge
$v_{a_b} \stackrel{1}{\rightarrow} d$, and $v_{a_b}$ is the last
element in the sequence. Clearly, $v_{a_b}$ is the half of $v_K$
in $\CP_K$. It is a property of $Double$-procedure that either
$v_K$ or $v_{a_b}$ is unmarked in $\CP_K$.

\vspace{3mm}\noindent{\bf (D7)} Similar to the proof of (D6).

\end{proof}

By (D1) we have all vertices $v^{(2^0)},\ldots,v^{(2^N)}$ in
$\CP_{t-1}$ (where $n\ge 2$). The value of a new auxiliary vertex
must be strictly greater than $2^{N}$ and to introduce a new
auxiliary vertex we need a vertex $v$ for which there are edges $v
\stackrel{1}{\rightarrow} v^{(2^0)}, \ldots, v
\stackrel{1}{\rightarrow} v^{(2^N)}$. But by property (D5) any
vertex connected to $d = v^{(2^N)}$ has $2^N$ as the lowest
summand of its power, so it cannot be connected to the vertices
$v^{(2^{0})}, \ldots, v^{(2^{N-1})}$. Thus, processing of $v_t$
cannot introduce a new auxiliary vertex.

\end{proof}

The estimate $|V(\CP')| \le |V(\CP)|+1$ in the statement of
Proposition \ref{pr:reduction_and_vertices} cannot be
further improved. Figure \ref{fi:new_vertex} gives an
example when $|V(\CP')| = |V(\CP)|+1$.

\begin{proposition} ({\em Complexity of reduction})\label{pr:reduction_complexity}
The complexity of Algorithm \ref{al:reduce} is $O(|V(\CP)|^3)$.
\end{proposition}

\begin{proof}
Denote by $m$ the number of edges in $\CP$. Observe that
from property (R2) it follows that $m \le n^2$.

We analyze each step in Algorithm \ref{al:reduce}. Trimming and
removing redundancies around zero requires $O(n+m) \le O(n^2)$
steps. The same time complexity is required for computing the
geometric order on $\CP$. The most complicated part is step E).
For each vertex $v_i$:
\begin{enumerate}
    \item[1)]
Removing redundancies requires at most $O(n)$ steps.
    \item[2)]
It takes linear time $O(n)$ to compare two $\CE$-values and it
will take $O(n \log n)$ steps to find a position of $v_i$ in the
current ordered set $C$.
    \item[3)]
Adding $v_i$ into $C$ takes a constant time $O(1)$ to perform.
    \item[4)]
Separation of $v_i$ in $C_i$ requires $O(n^2)$ steps.
\end{enumerate}
Therefore, processing of $v_i$ requires in the worst case $O(n^2)$
steps and processing of all vertices in $\CP$ requires $O(n^3)$
steps. Summing all up we get the result.

\end{proof}

\section{Computing normal forms}
\label{se:normal_form_computing}

In this section we show how to find normal forms of constant power circuits.

\begin{lemma}
For a constant power circuit $\CP$ one can check in time $O(|V(\CP)|^3)$
whether $\CP$ is proper, or not.
\end{lemma}

\begin{proof}
By definition $\CP$ is proper if and only if $\CE(v) \in \MN$ for every $v\in V(\CP)$.
This is implicitly checked in the reduction process that requires $O(|V(\CP)|^3)$
operations.
\end{proof}

\begin{theorem}\label{th:normal_cp_number}
There exists a procedure which for any $n \in\MN$ computes
the unique constant normal power circuit $\CP_n$ representing $n$
in time $O(\log_2 n \log_2 \log_2 n)$.
Furthermore, the circuit $\CP_n$ satisfies $|V(\CP_n)| \le \lceil\log_2 n\rceil +2$.
\end{theorem}

\begin{proof}
We construct a circuit for $n$ explicitly.
Put $k = \lceil\log_2 n\rceil$ and $V = \{0,2^0,2^1,\ldots,2^{k}\}$.
Define the set of labeled directed edges on $V$
    $$E = \{2^q \stackrel{\varepsilon}{\rightarrow} 2^s \mid \mbox{the compact sum for $q$ involves } \varepsilon 2^s\}
    \cup \{2^0\stackrel{1}{\rightarrow} 0\}.$$
If $\varepsilon_1 2^{q_1}+\ldots+\varepsilon_k 2^{q_k}$
is a compact binary sum for $n$ then put $M = \{2^{q_1},\ldots,2^{q_k}\}$ and $\nu(2^{q_i}) = \varepsilon_i$.
Trim the obtained power circuit.
Denote the constructed circuit by $\CP_n$.
It follows from construction that the obtain power circuit $\CP_n$ is normal and $\CE(\CP_n) = n$.
Also, it follows from the construction that $|V(\CP_n)| \le k +2$
and $|E| \le k\log_2 k$. Furthermore, it is straightforward to find the set $E$.
Thus, the time complexity of the described procedure is $O(k \log_2 k)$.
\end{proof}

\begin{theorem}\label{th:normal_circuit}
There exists an algorithm which for
a given constant proper power circuit $\CP$
computes the unique (up to isomorphism)
equivalent proper normal power circuit $\CP'$ in time $O(|V(\CP)|^3)$.
Furthermore, $|V(\CP')| \le 2|V(\CP)|$.
\end{theorem}

\begin{proof}
{\bf (Step A)} Compute $\CP' = \Reduce(\CP)$.
The reduction procedure orders the set
$V(\CP') = \{v_1,\ldots,v_n\}$ so that $\CE(v_i) < \CE(v_{i+1})$.
Also, it provides us with a sequence
$d_1,\ldots,d_{n-1}$ of $0$'s and $1$'s satisfying $d_i = 1$ if and only if
$2\CE(v_{i}) = \CE(v_{i+1})$.
Since $\CP'$ is reduced, it follows that the sum
$\CE(\CP') = \sum_{v\in M} \nu(v) \CE(v)$
is reduced
and for every vertex $v\in V(\CP')$ the sum $\sum_{e\in Out_{v}} \mu(e)\CE(\beta(e))$
is reduced. Our goal is to make these sums compact.
By Lemma \ref{le:binary_to_compact} to make these sums
compact we might need to introduce doubles for some vertices in $V(\CP)$.
We do it next.

{\bf (Step B)}
Let $\{v_{a_1},\ldots,v_{a_n}\}$ be a geometric order on $V(\CP)$.
For every vertex $v_{a_i}$ (from smaller indices to larger) such that $d_{a_i} = 0$
introduce its double, i.e., add a new vertex $v_{a_i}'$
and add edges so that $\CE(v_{a_i}') = 2\CE(v_{a_i})$
as described in Algorithm \ref{al:separation}.
It very important to note that Algorithm \ref{al:separation}
never performs step B.2) (and hence does not introduce new auxiliary vertices)
because the vertex $v^{2^N}$ in the description of Algorithm \ref{al:separation}
is a double of some vertex $v_{a_j}$ and it is already introduced.

{\bf (Step C)}
Next we use the procedure described in Lemma \ref{le:compact_sum}
to make sure that for every vertex $v \in V(\CP)$
the binary sum $\sum_{e\in Out_{v}} \mu(e)\CE(\beta(e))$
is compact.
Since the doubles were introduced to $\CP'$
it follows from Lemma \ref{le:binary_to_compact} that this can be done.

{\bf (Step D)}
To make the sum $\CE(\CP) = \sum_{v\in M} \nu(v) \CE(v)$ compact we
change $M$ and $\nu$ as described in Lemma \ref{le:compact_sum}.
By Lemma \ref{le:binary_to_compact} we can do that.

{\bf (Step E)}
Finally, we trim the obtained power circuit and output the result.l

The reduction step is the most time consuming step which requires $O(|V(\CP)|^3)$
steps. Hence the claimed bound on time complexity.
\end{proof}

\section{Elementary operations over power circuits}
\label{se:operations}

In this section we show how to efficiently perform arithmetic
operations over power circuits.

\subsection{Addition and subtraction}
\label{se:addition}

Let $\CP_1$ and $\CP_2$ be two circuits. The following algorithm
computes a circuit $\CP_+$ such that $\CT_{\CP_{+}} = \CT_{\CP_1} +
\CT_{\CP_2}$ over $\mathbb{Z}$ (or $\mathbb{R}$).

\begin{algorithm} \label{al:sum_pp}{\em (Sum of circuits)}
        \\{\sc Input.}
Circuits $\CP_1 = (\CP_1,M_1,\mu_1,\nu_1)$ and $\CP_2 =
(\CP_2,M_2,\mu_2,\nu_2)$.
    \\{\sc Output.}
Circuit $\CP_+ = (\CP_+,M,\mu,\nu)$ such that $\CT_{\CP_{+}} = \CT_{\CP_1} +
\CT_{\CP_2}$ over $\mathbb{Z}$.\\
{\sc Computations.}
\begin{enumerate}
 \item[A)] Let $\CP_+$ be a disjoint union of graphs $\CP_1$ and $\CP_2$.
 \item[B)] Put $M = M_1 \cup M_2$.
 \item[C)] Define a function $\nu$ on $M$ such that $\nu|_{M_{1}} =
 \nu_{1}$ and $\nu|_{M_{2}} = \nu_{2}$.
 \item[D)] Define a function $\mu$ on $E(\CP_+)$ such that $\mu|_{E(\CP_1)} =
 \mu_1$ and $\mu|_{E(\CP_2)} = \mu_2$
 \item[E)] Return $\CP_+ = (\CP_+,M,\mu,\nu)$.
\end{enumerate}
\end{algorithm}

\begin{figure}[h]
\centerline{ \includegraphics[scale=0.5]{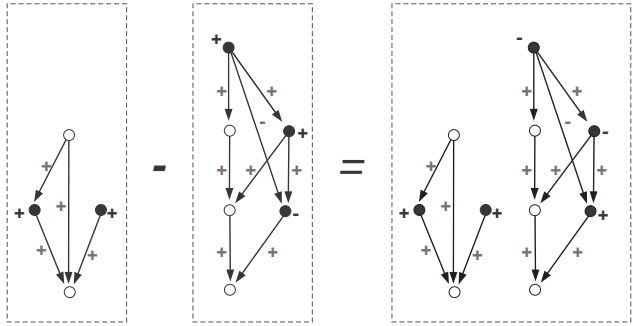} }
\caption{\label{fi:difference} Difference of circuits.}
\end{figure}

\begin{proposition} \label{pr:complexity_add}
Let $\CP_1$ and $\CP_2$ be power circuits.  Then
  \begin{enumerate}
 \item [1)]  $\CT_{\CP_{+}} = \CT_{\CP_1} +
\CT_{\CP_2}$ over $\mathbb{Z}$,
 \item [2)] Algorithm \ref{al:sum_pp} computes $\CP_+$
in linear time $O(|\CP_1| + |\CP_2|)$.
  \item [3)] Moreover,  the size of $\CP_+ = (\CP_+,M,\mu,\nu)$ is bounded as follows:
    \begin{itemize}
    \item  $|V(\CP_+)| = |V(\CP_1)| + |V(\CP_2)|$,
     \item $|E(\CP_+)| = |E(\CP_1)| + |E(\CP_2)|$,
     \item $|M| = |M_1| + |M_2|$.
 \end{itemize}
\end{enumerate}
\end{proposition}

\begin{proof}
Straightforward from the construction of  $\CE(\CP_+)$ in Algorithm \ref{al:sum_pp}.
\end{proof}

A similar result holds for subtraction $-$. To compute $\CP_- = \CP_1 - \CP_2$ one can modify Algorithm
\ref{al:sum_pp} as follows. At step C) instead of putting
$\nu|_{M_{2}} = \nu_{2}$ put $\nu|_{M_{2}} = -\nu_{2}$. Clearly,
for the obtained circuit $\CP_-$ the equality $\CT_{\CP_{-}} = \CT_{\CP_1} -
\CT_{\CP_2}$ over $\mathbb{Z}$, as wells as the complexity and size estimates of Lemma
\ref{pr:complexity_add} hold.

Sometimes we refer to the circuits $\CP_+$ and $\CP_-$ as  $\CP_1 + \CP_2$  and $\CP_1 - \CP_2$, correspondingly.

\subsection{Exponentiation}

Let $\CP$ be a power circuit. The next algorithm produces a
circuit $\CP'$ such that $\CT_{\CP'} = 2^{\CT_\CP}$.

\begin{algorithm}\label{al:exponent}({\em Exponentiation in base  $2$})
      \\{\sc Input.}
A circuit $\CP = (\CP,M,\mu,\nu)$.
    \\{\sc Output.}
A circuit $\CP' = (\CP', M',\mu',\nu')$ such that $\CT_{\CP'} = 2^{\CT_\CP}$.
    \\{\sc Computations:}
\begin{enumerate}
    \item[1)] Construct a graph $\CP'$ as follows:
\begin{itemize}
 \item Add a new unmarked vertex $v_0$ into the graph $\CP$.
    \item
For each $u \in M$ add an edge $e = (v_0 \rightarrow u)$.
\end{itemize}
    \item[2)] Put $M = \{ v_0 \}$.
\item [3)] Define   $\nu'(v_0) = 1$.
    \item[4)] Extend $\mu$ to $\mu'$  defining $\mu'$ on new edges $( v_0,u)$ by  $\mu'( v_0,u) = \nu(u)$.
 \item [5)] Output $(\CP,M,\mu,\nu)$.
\end{enumerate}
\end{algorithm}

See Figure \ref{fi:power_of_two} for an example.

\begin{figure}[htbp]
\centerline{
\includegraphics[scale=0.5]{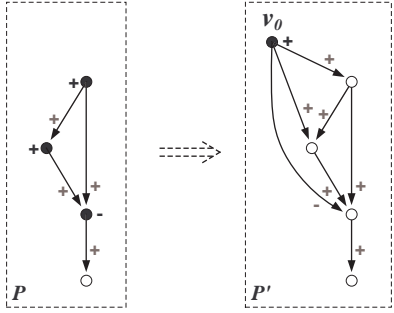} }
\caption{\label{fi:power_of_two} Exponentiation in base  $2$. }
\end{figure}

\begin{proposition} \label{pr:complexity_exponent}
Let $\CP$ be a power circuit. Then
  \begin{enumerate}
 \item [1)] $\CT_{\CP'} = 2^{\CT_\CP}$,
 \item [2)]  Algorithm
\ref{al:exponent} computes $\CP'$ in linear time $O(|\CP|)$.
  \item [3)] Moreover,  the size of $\CP' = (\CP', M',\mu',\nu')$ is bounded as follows:
    \begin{itemize}
    \item  $|V(\CP')| = |V(\CP)| + 1$,
     \item $|E(\CP')| \le |E(\CP)| + |V(\CP)|$,
     \item $|M'| = 1$.
 \end{itemize}
\end{enumerate}

\end{proposition}

\begin{proof}
Recall that $\CT_\CP = \sum_{v\in M} \nu(v) t_v$. Therefore,
$2^{\CT_\CP} = 2^{\sum_{v\in M} \nu(v) t_v}$ which is exactly
the term  $\CT_{\CP'}$. The other statements follow from the
constructions in  Algorithm \ref{al:exponent}.

\end{proof}

Sometimes we refer to the circuit $\CP'$ as $2^\CP$.

\subsection{Multiplication}
\label{se:multiplication}

Let $\CP_1$ and $\CP_2$ be two power circuits. In this section we
construct a power circuit  $\CP_\ast$ such
that $\CT_{\CP_\ast} = \CT_{\CP_1} \cdot \CT_{\CP_2}$.

\begin{algorithm} \label{al:product_pp}{\em (Product of circuits)}\\
{\sc Input.} Circuits $\CP_1$ and $\CP_2$.\\
{\sc Output.} A circuit $\CP_\ast$ such that $\CT_{\CP_\ast} = \CT_{\CP_1} \cdot \CT_{\CP_2}$ in any exponential ring $R$.\\
{\sc Computations.}
\begin{enumerate}
    \item[A)]
Apply Algorithm \ref{al:marked_origins} to get power circuits  $\CP_1'$ and $\CP_2'$, which are equivalent to $\CP_1$ and $\CP_2$ and where all marked vertices are sources.
    \item[B)]
Construct $\CP = (V(\CP),E(\CP))$, where
$$V(\CP) ~=~ (V(\CP_1') \smallsetminus M(\CP_1')) ~~\cup~~ (V(\CP_2') \smallsetminus M(\CP_2')) ~~\cup~~ M(\CP_1') \times M(\CP_2').$$
and $E(\CP)$ contains edges of three types:
\begin{enumerate}
\item[1)] for each edge $v_1 \stackrel{x}{\rightarrow} v_2$ in
$\CP_i'$ such that $v_1,v_2 \in V(\CP_i') \smallsetminus
M(\CP_i')$ ($i=1,2$) add an edge $v_1 \stackrel{x}{\rightarrow}
v_2$ into $\CP$;
    \item[2)]
for each edge $v_1 \stackrel{x}{\rightarrow} v_2$ in
$\CP_1'$, where $v_1$ is marked and $v_2$ is not, and for each
vertex $v_3 \in M(\CP_2')$ add an edge $(v_1,v_3)
\stackrel{x}{\rightarrow} v_2$ into $\CP$;
    \item[3)]
for each edge $v_1 \stackrel{x}{\rightarrow} v_2$ in
$\CP_2'$, where $v_1$ is marked and $v_2$ is not, and for each
vertex $v_3 \in M(\CP_1')$ add an edge $(v_3,v_1)
\stackrel{x}{\rightarrow} v_2$ into $\CP$.
\end{enumerate}
    \item[C)]
Put $M = M(\CP_1') \times M(\CP_2')$ and for each $v = (v_1,v_2)
\in M$ put $\nu(v) = \nu(v_1) \nu(v_2)$.
    \item[D)]
Output the obtained circuit $(\CP,M,\mu,\nu)$.
\end{enumerate}
\end{algorithm}

\begin{figure}[htbp] \centerline{
\includegraphics[scale=0.5]{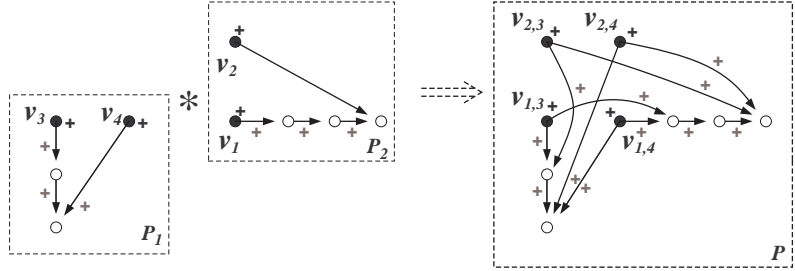} }
\caption{\label{fi:figure_16} Circuit multiplication.}
\end{figure}

\begin{proposition}
Let $\CP_1$ and $\CP_2$ be two power circuits and $\CP_\ast$ obtained from them by Algorithm \ref{al:product_pp}. Then:

 \begin{itemize}
\item $\CT_{\CP_\ast} = \CT_{\CP_1} \cdot \CT_{\CP_2}$,
\item $$|V(\CP_\ast)| \le |V(\CP_1)| + |V(\CP_2)| + |M(\CP_1)| \cdot |M(\CP_2)|$$
and
$$|E(\CP_\ast)| \le 2(|V(\CP_1)| + |V(\CP_2)| ) \cdot |V(\CP_1)| \cdot |V(\CP_2)|.$$
\item  Algorithm \ref{al:product_pp} computes $\CP_\ast$ in at
most cubic time $O(|V(\CP_1)| \cdot |V(\CP_2)| \cdot (|V(\CP_1)| +
|V(\CP_2)|) )$.
\end{itemize}
\end{proposition}

\begin{proof}
Since
    $$\CT(\CP_1) =  \sum_{v_i \in M_1} \nu(v_i) 2^{\left(\sum_{e \in Out_{v_i}} \mu(e) t_{\beta(e)}\right)}$$
and
    $$\CT(\CP_2) = \sum_{v_j \in M_2} \nu(v_j) 2^{\left(\sum_{e \in Out_{v_j}} \mu(e) t_{\beta(e)}\right)},$$
we get
$$\CT(\CP_1) \CT(\CP_2) =
 \sum_{v_i\in M_1,~v_j\in M_2} \nu(v_i)\nu(v_j) 2^{\left(\sum_{e
 \in Out_{v_i}} \mu(e) t_{\beta(e)} + \sum_{e \in
 Out_{v_j}} \mu(e) t_{\beta(e)}\right)}$$
$$= \sum_{(v_i,v_j) \in M = M_1 \times M_2} \nu((v_i,v_j)) 2^{\sum_{e \in
 Out_{(v_i,v_j)}} \mu(e) t_{\beta(e)}} = \CT(\CP_\ast).$$

To show that estimates for $V(\CP_\ast)$ and the time-complexity
hold we analyze Algorithm \ref{al:product_pp} step by step. By Lemma \ref{le:source}  Algorithm \ref{al:marked_origins}  is linear time and the following  estimates of the sizes hold:
$$|V(\CP_i')| \le |V(\CP_i)| + |M(\CP_i)| ~\mbox{ and }~ |M(\CP_i')| = |M(\CP_i)|$$
(where $i=1,2$). Hence, the time complexity of this step is at most
$O(|M(\CP_1)|+|V(\CP_1)| + |M(\CP_2)|+|V(\CP_2)|)$. On the next
two steps (B and C) we construct the graph $\CP$ in a very
straightforward way, so the complexity of these steps is the size
of $\CP$. By construction of $\CP$ we have
$$V(\CP_\ast) = (V(\CP_1') \smallsetminus M(\CP_1')) ~\cup~ (V(\CP_2') \smallsetminus
 M(\CP_2')) ~\cup~ M(\CP_1') \times M(\CP_2')$$
and the claimed estimate on $|V(\CP)|$ holds. Clearly, $|E(\CP)'|
\le |E(\CP_1)| + |E(\CP_2)| + |M(\CP_1)| \cdot |M(\CP_2)| \cdot
(|V(\CP_1)| + |V(\CP_2)|) \le 2 \cdot |V(\CP_1)| \cdot |V(\CP_2)|
\cdot (|V(\CP_1)| + |V(\CP_2)|)$. This gives the claimed
upper bound on the time complexity of Algorithm
\ref{al:product_pp}.

\end{proof}

Sometimes we denote the circuit $\CP_\ast$ constructed above by  $\CP_1 \ast \CP_2$.

\subsection{Multiplication and division by a power of two}
\label{se:multiplication_power}

Let $\CP_1$ and $\CP_2$ be power circuits. In this section we
present a procedure for constructing circuits $\CP_\bullet$ and
$\CP_\circ$ such that
    $$\CT(\CP_\bullet) =  \CT(\CP_1) \cdot 2^{\CT(\CP_2)} \mbox{ and } \CT(\CP_\circ) = \CT(\CP_1) \cdot2^{-\CT(\CP_2)}.$$
Observe that both $\CP_\bullet$ and $\CP_\circ$ can be constructed using operations above.
However, we present different  more efficient
procedures to build the required circuits.

\begin{algorithm} \label{al:multiplication_power2}{\em (Multiplication by a power of 2)}\\
{\sc Input.} Circuits $\CP_1$ and $\CP_2$.\\
{\sc Output.} A circuit $\CP_\bullet$ such that $\CT(\CP_\bullet) =  \CT(\CP_1) \cdot 2^{\CT(\CP_2)}$.\\
{\sc Computations.}
\begin{enumerate}
\item[A)] Construct the circuit  $\CP_1'$ which is equivalent to $\CP$ and where all marked vertices are sources.  Assume that $\CP_1' =
(\CP_1',M_1',\mu_1',\nu_1')$ and $\CP_2 =
(\CP_2,M_2,\mu_2,\nu_2)$.

\item[B)] Define $\CP_\bullet = (\CP_\bullet,M,\mu,\nu)$ as follows:

\begin{enumerate}
\item[1)] $\CP_\bullet$ is a disjoint union of $\CP_1'$ and
$\CP_2$.

\item[2)] For each $v_1 \in M_1$ and each $v_2 \in M_2$ add an
edge $v_1 \stackrel{\nu(v_2)}{\longrightarrow} v_2$ into
$\CP_\bullet$.

\item[3)] $M = M_1$ and $\nu = \nu_1$.

\end{enumerate}

\item[C)] Output $\CP_\bullet$.

\end{enumerate}
\end{algorithm}

Of course, the operation $x\cdot 2^{-y}$ can be expressed via subtraction and $x\cdot2^y$. However, we need a proper power circuit representation of an integer $x\cdot 2^{-y}$.

\begin{algorithm} \label{al:division_power2}{\em (Division by a power of 2)}\\
{\sc Input.} Constant power circuits $\CP_1$ and $\CP_2$.\\
{\sc Output.} A constant circuit $\CP_\circ$ such that $\CE(\CP_\circ) = \CE(\CP_1) 2^{-\CE(\CP_2)}$ and this is proper.\\
{\sc Computations.}
\begin{enumerate}
\item[A)] Let  $\CP_1'$ be a reduced constant power circuit equivalent to $\CP_1$ where all marked vertices are sources. Assume that
$\CP_1' = (\CP_1',M_1',\mu_1',\nu_1')$ and $\CP_2 =
(\CP_2,M_2,\mu_2,\nu_2)$.

\item[B)] Define $\CP_\circ$ to be $(\CP_\circ,M,\mu,\nu)$ where:

\begin{enumerate}
\item[1)] $\CP_\circ$ is a disjoint union of $\CP_1'$ and $\CP_2$.

\item[2)] For each $v_1 \in M_1$ and each $v_2 \in M_2$ add an
edge $v_1 \stackrel{-\nu(v_2)}{\longrightarrow} v_2$ into
$\CP_\circ$.

\item[3)] $M = M_1$ and $\nu = \nu_1$.

    \item[4)]
Collapse zero vertices in $\CP_\circ$ (there are at least $2$ of them,
one coming from $\CP_1$ and the other from $\CP_2$).
\end{enumerate}

\item[C)] Output $\CP_\circ$.

\end{enumerate}
\end{algorithm}

\begin{proposition}\label{pr:complexity_divpower2}
Let $\CP_1= (\CP_1,M_1,\mu_1,\nu_1)$ and $\CP_2 =
(\CP_2,M_2,\mu_2,\nu_2)$ be circuits, $\CP_\bullet = \CP_1 \bullet
\CP_2$, and $\CP_\circ = \CP_1 \circ \CP_2$. Then
\begin{itemize}
    \item[1)]
$\CE(\CP_\bullet) = \CE(\CP_1) 2^{\CE(\CP_2)}$ and $\CE(\CP_\circ)
= \frac{\CE(\CP_1)} {2^{\CE(\CP_2)}}$;
    \item[2)]
$|V(\CP_\bullet)|,|V(\CP_\circ)| \le |V(\CP_1)|+|V(\CP_2)|+|M_1|$.
    \item[3)]
The time complexity of Algorithm \ref{al:multiplication_power2} is
bounded from above by $O(|\CP_1|+|\CP_2|+|M_1|\cdot|M_2|)$.
    \item[4)]
The time complexity of Algorithm \ref{al:division_power2} is
bounded from above by $O(|V(\CP_1)|^3+|\CP_2|+|M_1|\cdot|M_2|)$.
\end{itemize}
\end{proposition}

\begin{proof}
Straightforward to check.
\end{proof}

We already pointed out that the operation $\CP_1 \circ \CP_2$
is not defined for all pairs of power circuits
$\CP_1 = (\CP_1, \mu_1, M_1,\nu_1)$, $\CP_2 = (\CP_2, \mu_2, M_2,\nu_2)$
because the value $\CE(\CP_1) \cdot 2^{-\CE(\CP_2)}$ is not always
an integer.
We can naturally extend the domain of definition of $\circ$
to the set of all pairs $\CP_1,\CP_2$ by rounding
the value of $\CE(\CP_1) \cdot 2^{-\CE(\CP_2)}$.

Our algorithms do not become less efficient
if we use $\circ$ with rounding. Indeed, if
    $$\CE(\CP_1) = \sum_{v\in M_1} \nu_1(v) \CE(v)~ \mbox{ where }~ \CE(v) = 2^{\left(\sum_{e \in Out_{v}} \mu_1(e) \CE(\beta(e))\right)}$$
then
    $$\CE(\CP_1) \cdot 2^{-\CE(\CP_2)} = \sum_{v\in M_1} \nu_1(v) 2^{\left(\sum_{e \in Out_{v}} \mu_1(e) \CE(\beta(e))\right)-\CE(\CP_2)}.$$
To round up the value of $\CE(\CP_1) \cdot 2^{-\CE(\CP_2)}$ it is sufficient to remove
all vertices $v$ from $M_1$ such that $\sum_{e \in Out_{v}} \mu_1(e) \CE(\beta(e)) < \CE(\CP_2)$.
That can be done in polynomial time by Proposition \ref{pr:reduction_complexity}.

\subsection{Ordering}

Clearly, $\CE(\CP_1)<\CE(\CP_2)$ if and only if $\CE(\CP_1-\CP_2)<0$. Therefore,
to compare values of constant power circuits $\CP_1$ and $\CP_2$
it is sufficient to  compare a value of a circuit $\CP_1-\CP_2$ with $0$.
For a constant power circuit $\CP$ define
$$Sign(\CP) =
 \left\{
 \begin{array}{ll}
 -1, & \mbox{if } \CE(\CP)<0;\\
 0, & \mbox{if } \CE(\CP)=0;\\
 1, & \mbox{if } \CE(\CP)>0.\\
 \end{array}
 \right.
$$\\

\begin{algorithm} \label{al:compare}({\em Sign of $\CE(\CP)$})\\
{\sc Input.} A circuit $\CP$.\\
{\sc Output.} $Sign(\CP)$.
\\
{\sc Computations.}
\begin{enumerate}
    \item[A)]
Let $\CP' = Reduce(\CP)$ and $C = \{v_1,\ldots,v_k\}$ be a
sequence of vertices produced by Algorithm \ref{al:reduce} such
that $\CE(v_i) < \CE(v_j)$ whenever $1\le i<j\le k$.
    \item[B)]
If $\CP'$ is trivial then output $0$.
    \item[C)]
Find the marked vertex $v_i$ in $\CP'$ with the greatest index
$i$.
    \item[D)]
Output $\nu(v_i)$.
\end{enumerate}
\end{algorithm}

\begin{proposition} \label{pr:complexity_compare_circ}
Let $\CP$ be a constant power circuit. Then  Algorithm
\ref{al:compare} computes $Sign(\CP)$ in time bounded from above by $O(|V(\CP)|^3)$.
\end{proposition}

\begin{proof}
Let $\CP'$ be the reduced power circuit equivalent to $\CP$  produced by Algorithm \ref{al:reduce},  and $C = \{v_1,\ldots,v_k\}$ be a
sequence of vertices produced by Algorithm \ref{al:reduce} such
that $\CE(v_i) < \CE(v_j)$ whenever $1\le i<j\le k$. Then
$$\CE(\CP) = \CE(\CP') = \sum_{v_j\in M} \nu(v_j) \CE(v_j) = \sum_{v_j\in M} \nu(v_j) 2^{\left(\sum_{e \in Out_{v_j}} \mu(e) \CE(\beta(e))\right)}$$
which is a reduced binary sum (see Section
\ref{se:elementary_properties}). By Proposition \ref{pr:circuit_compare}
$Sign(\CE(\CP)')$ is the coefficient of the greatest power of
$2$,  which is $\nu(v_i)$, where $i$ is the greatest index such that
$v_i\in M$. Hence  $Sign(\CE(\CP)) = \nu(v_i)$ as claimed.

By Proposition \ref{pr:reduction_complexity}, the reduction
process performed by algorithm \ref{al:reduce} has time-complexity
$O(|V(\CP)|^3)$. Once $\CP$ is reduced, it is immediate to find
the value of $\nu(v_i)$. Thus, $O(|V(\CP)|^3)$ is an upper bound
for time-complexity of Algorithm \ref{al:compare}.
\end{proof}

\section{Exponential algebra on power circuits}
\label{se:exponential_algebra_circuits}

Fix a language
    $$\CL = \{+,-,\ast, x\cdot 2^y, x\cdot2^{-y},\le, 0, 1\},$$
its sublanguage $\CL_0$, which is obtained from $\CL$ by removing the multiplication $\ast$; and structures
    $$\MZ_\CL = \gp{\MZ;+,-,\ast, x\cdot 2^y, x\cdot2^{-y}, \le,1}$$
and
    $$\tilde Z = \MZ_{\CL_0} = \gp{\MZ;+,-, x\cdot 2^y, x\cdot2^{-y}, \le,1}.$$
In this section we show that there exists an algorithm
that for every algebraic $L$-circuit $C$ finds an
equivalent standard power circuit $\CP$, or equivalently,
there exists an algorithm which for every term $t$ in
the language $\CL$ finds a power circuit $C_t$ which
represents a term equivalent to the term $t$ in $\MZ_\CL$.
Moreover, if the term $t$ is in the language $\CL_0$ then
the algorithm computes the circuit $C_t$ in linear time in
the size of $t$.  For integers and closed terms in $\CL_0$
one can get much stronger results. Let $\CC_{norm}$
be the set of all constant normal power circuits
(up to isomorphism). We show that if  $t(X)$ is a
term in $\CL_0$ and $\eta:X \to \MZ$ an assignment of
variables, then there exists an algorithm which determines
if $t(\eta(X))$ is defined in $\MZ_\CL$ (or $\tilde Z$)  or not;
and if defined it then produces the normal
circuit $\CP_t $ that presents the number $t(\eta(X))$ in
polynomial time. At the end of the section we prove that
the  quantifier-free theory of the structure  $\tilde{Z}$
with all the constants from $\MZ$ in the language is decidable in polynomial time.

\subsection{Algebra of power circuits}

We have mentioned in Introduction that every term $t$ in the language $\CL$
can be realized in $\tilde Z$ by an algebraic $\CL$-circuit. In this section we show that every such term $t$ also
can be realized in $\tilde Z$ by a power circuit $\CP_t$. Furthermore, we show that
if $t$ does not involve multiplication, then the circuit $\CP_t$ can be computed in polynomial time
in the size of $t$ (which may not be true if $t$ involves multiplications).

Let $\CC$ be the set of all power circuits in variables from a set $X = \{x_1, x_2, \ldots, \}$.
Recall, that two circuits $\CP_1, \CP_2 \in \CC$ are equivalent ($\CP_1 \sim \CP_2$)
if the terms $\CT_{\CP_1}$ and $\CT_{\CP_2}$ define the same function in  $\tilde Z$.
In Section \ref{se:operations} we defined  operations $+,-,\ast, x\cdot 2^y, x\cdot2^{-y}$ on
power circuits. It is easy to see
from the construction that these operations are compatible with the equivalence
relation $\sim$, so they induce the corresponding operations on the quotient set $\CC/\sim$,
forming an algebraic $\CL$-structure
$$\CC = \langle \CC/\sim; +,-,\ast, x\cdot 2^y, x\cdot2^{-y},0,1\rangle$$
where we interpret the constants $0, 1$ by the equivalent classes of the  normal
power circuits with the values $0$ and $1$.

To clarify the algebraic structure of $\CC$ we need the following.
Let $\CT_\CL$ be the set of all terms in the language $\CL$.
Two terms  $t_1$ and $t_2$ are termed {\em equivalent} ($t_1 \sim t_2$) if they define the same
functions on $\tilde Z$. The quotient set $\CT_\CL/\sim$ can be naturally identified
with the set $\CF_\CL$ of all {\em term functions} induced  by terms from $\CT_\CL$ in $\tilde Z$. Obviously, the operations in $\CC$ are precisely the same as the corresponding  operations over the term functions in $\CF_\CL$.

Denote by $C_0, C_1$ and $C_x$ some standard  circuits that realize the terms $0,1$ and a variable $x \in X$.
Define a map
$$\tau:\CT_\CL \to \CC$$
 by induction on complexity of the terms:
\begin{itemize}
    \item
if $t \in \{0,1\}\cup X$ then $\tau(t) = C_t$;
    \item
if $t = f(t_1,t_2)$ where $t_1,t_2$ are terms and $f$ is an operation from $\CL$ then
the circuit $\tau(t) = f(\tau(t_1),\tau(t_2))$ is obtained from $\tau(t_1)$ and $\tau(t_2)$
as  described in Section \ref{se:operations}.
\end{itemize}
The next proposition immediately follows from the construction.

\begin{proposition}
The following hold:
\begin{itemize}
    \item[(1)]
$\tau$ induces an isomorphism of the algebraic structures
$$\tau: \langle \CT_\CL/\sim ;+,-,\ast, x\cdot 2^y, x\cdot2^{-y},0,1\rangle \to \langle \CC/\sim ;+,-,\ast, x\cdot 2^y, x\cdot2^{-y},0,1 \rangle.$$
    \item[(2)]
Let $t \in \CT_\CL$ and $\CP = \tau(t)$. Then the terms $t$ and $\CT_P$ are equivalent in $\MZ_\CL$.
\end{itemize}
\end{proposition}

\begin{corollary}
There is an algorithm that for every algebraic $L$-circuit $C$ finds an equivalent standard power circuit $\CP$.
\end{corollary}

Let $\CT_{\CL_0}$ be a subset of $\CT_\CL$ consisting of terms in the language $\CL_0$.
We prove now that the restriction  of $\tau$ on $\CT_{L_0}$ is linear
time computable in the size of an input term $t$ (the number  $|t|$ of operations that occur in $t$).

\begin{theorem} \label{th:theoremB}
Given $t \in \CT_{L_0}$ it requires  at most $O\rb{|t|}$
steps to compute $\CP = \tau(t)$.
Furthermore, $|M(\CP)| \le |t|+1$, $|V(\CP)| \le 2|t|+2$, and every marked vertex in $\CP$
is a source.
\end{theorem}

\begin{proof}
Induction on complexity of the term $t$.
The terms $0$, $1$, and $x \in X$ do not involve any operations, the corresponding circuits $C_0, C_1, C_x$ satisfy the conditions
$|M(\CP)| \le 1$, $|V(\CP)| \le 2$, and have the property that every marked vertex is a source.
Now, assume that the statement holds for terms $t_1$ and $t_2$.
Let $t = f(\tau_1,\tau_2)$, where $f$ is an operation from $\CL_0$ and $\CP_1 = \tau(t_1)$, $\CP_2 = \tau(t_2)$.
Let $\CP = f(\CP_1,\CP_2)$ constructed by the appropriate algorithm from Section \ref{se:operations}.
Since every vertex in $\CP_2$ is a source it immediately follows from
Algorithms \ref{al:sum_pp} and \ref{al:multiplication_power2} that
    $$|M(\CP)| \le |M(\CP_1)| + |M(\CP_2)| \mbox{ and } |V(\CP)| \le |V(\CP_1)| + |V(\CP_2)|$$
and every vertex in $\CP$ is a source. Therefore, $|M(\CP)| \le |t_1|+1+|t_2|+1 = |t|+1$
and $|V(\CP)| \le 2|t_1|+2+2|t_2|+2 = |t|+2$. Moreover, the circuit $P$ in both cases is computed in linear time in $|t|$.
\end{proof}

In contrast to Theorem \ref{th:theoremB} we construct in Section \ref{se:with-multiplication} a sequence of terms $\{t_i\}$
with multiplication in the language such that the size of $\tau(t_i)$ grows exponentially.

\subsection{Power representation of integers}

Let $\CC_{norm}$ be the set of all constant normal power circuits
up to isomorphism (so $\CC_{norm}$ consists of the equivalence classes of isomorphic normal power circuits).
Every operation $f \in \CL$ induces a similar operation $f$ on $\CC_{norm}$ defined  for
$\CP_1, \CP_2 \in  \CC_{norm}$ as  $(\CP_1,\CP_2) \to Normal(f(\CP_1,\CP_2))$.
Define a map $\lambda: \mathbb{Z} \to \CC_{norm}$ such that $\lambda(n)$ is
the the unique (up to isomorphism) normal power circuit representing $n \in \mathbb{Z}$.
The next proposition follows directly from the definition of $\lambda$ and the results on normal power circuits.

\begin{proposition}
The following hold:
\begin{itemize}
    \item
the map $\lambda$ defines an isomorphism of $\CL$-structures
    $$\lambda: \langle \mathbb{Z}; +,-,\ast, x\cdot 2^y, x\cdot2^{-y},0,1\rangle \to
\langle \CC_{norm}; +,-,\ast, x\cdot 2^y, x\cdot2^{-y},0,1\rangle.$$
    \item
If $t$ is a closed term in $\CL$ which gives a number $n \in \mathbb{Z}$
then $\lambda(n) = Norm(\tau(t))$.
\end{itemize}
\end{proposition}

Let $L_0 = \{+,-,x\cdot 2^y, x\cdot2^{-y}\}$ and $\CT_{L_0}$ as above.
The next algorithm solves the term realization problem for $\CC_{norm}$.

\begin{algorithm} \label{al:term_realization}{\em (Term realization for $\CC_{norm}$)}\\
    {\sc Input.}
Let $t(x_1, \ldots,x_k) \in \CT_{L}$ be a term in variables $X_k = \{x_1, \ldots,x_k\}$  and $\eta:X_k \to \CC_{norm}$ an assignment of variables. \\
    {\sc Output.}
A circuit $\CP_t = t(\eta(X_k))$ if it is defined in $\MZ_\CL$. $Failure$ otherwise.\\
    {\sc Computations.}
\begin{enumerate}
    \item[(A)]
For every subterm $u$ of $t$ compute a reduced power circuit $\CP_{u}'$ realizing $u$ as follows:
\begin{enumerate}
    \item
If $u$ is a term $0$ then $\CP_{u}' = C_0$.
    \item
If $u$ is a term $1$ then $\CP_{u}' = C_1$.
    \item
If $u$ is a term $x \in X$ then $\CP_{u}' = \eta(x)$.
    \item
If $u$ is a term $u = f(u_1,u_2)$ where $f$ is an operation from $\CL$ then
apply Algorithm \ref{al:sum_pp} or Algorithm \ref{al:multiplication_power2} to circuits
$\CP_{u_1}'$ and $\CP_{u_2}'$ (we assume they are already constructed).
Reduce and denote the result by $\CP_u'$.
\end{enumerate}
If $\CP_u'$ does not represent an integer then output $Failure$.
    \item[(B)]
Compute the normal power circuit $\CP_t$ equivalent to $\CP_t'$ (use Theorem \ref{th:normal_circuit}).
    \item[(E)]
Output the $\CP_t$.
\end{enumerate}
\end{algorithm}

We summarize the discussion above in the following

\begin{theorem} \label{th:theoremB}
Let $t \in \CT_{\CL}$ and $\eta:X \to \CC_{norm}$ an assignment of variables.
Algorithm \ref{al:term_realization} determines whether  $t(\eta(X))$ is defined in $\MZ_\CL$  or not;
and if  defined  then it produces the normal
circuit $\CP_t$ which represent the number $t(\eta(X))$.
\end{theorem}

The next result shows that the Algorithm \ref{al:term_realization} is of polynomial time on terms from $\CT_{L_0}$.
For a term $t(X) \in \CT_{\CL_0}$ and a variable $x\in X$ define $\sigma_x(t)$ to be
the number of times the variables $x$ occurs  in $t$. Similarly,
define $\sigma_0(t)$ and $\sigma_1(t)$ to be the number of occurrences of
the constants $0$ and $1$ in $t$, respectively.

\begin{theorem}[Complexity of term realization]\label{th:term_realization_complexity}
Let $t(X) \in \CT_{L_0}$ and $\eta:X \to \CC_{norm}$ an assignment of variables.
Let $\CP_t$ be the output of Algorithm \ref{al:term_realization}.
Then
    $$|M(\CP_t)| \le \sum_{x\in X} \sigma_x(t) \cdot |M(\eta(x))| + \sigma_1(t).$$
    $$|V(\CP_t)| \le 2(|t|+1)\rb{\sum_{x\in X} \sigma_x(t) \cdot |V(\eta(x))|  + 2\sigma_1(t) + \sigma_0(t)}$$
and Algorithm \ref{al:term_realization} terminates in
    $$O\rb{ |t|^4 \cdot \rb{ \sum_{x\in X} \sigma_x(t) \cdot|V(\eta(x))|+ 2\sigma_1(t) + \sigma_0(t)}^3}$$
steps.
\end{theorem}

\begin{proof}
Following Algorithm \ref{al:term_realization} by induction on complexity of a subterm $u$ we prove that
\begin{equation}\label{eq:realization_mvertex_bound}
    |M(\CP_u')| \le \sum_{x\in X} \sigma_x(u) \cdot |M(\eta(x))| + \sigma_1(u).
\end{equation}
\begin{equation}\label{eq:realization_vertex_bound}
    |V(\CP_u')| \le (|u|+1)\rb{\sum_{x\in X} \sigma_x(u) \cdot |V(\eta(x))|  + 2\sigma_1(u) + \sigma_0(u)}.
\end{equation}
Indeed, the bounds (\ref{eq:realization_mvertex_bound}) and (\ref{eq:realization_vertex_bound})
clearly hold for the elementary terms $0$, $1$, and $x$.
If $u = f(u_1,u_2)$ where $f \in \CL_0$  then
one of the   Algorithms \ref{al:sum_pp} or \ref{al:multiplication_power2} (depending on $f$)
produces a circuit $\CP$ such that $\CE(\CP) = f(\CE(\CP_{u_1}'), \CE(\CP_{u_2}'))$.
For every such $f$ we have
\begin{equation}\label{eq:sum_marked}
    |M(\CP)| \le |M(\CP_{u_1}')|+|M(\CP_{u_2}')|
\end{equation}
\begin{equation}\label{eq:sum_vert}
     |V(\CP)| \le |V(\CP_{u_1}')|+|V(\CP_{u_2}')|+|M(\CP_{u_1}')|.
\end{equation}
Reducing the circuit $\CP$ to $\CP_{u}'$ does not increase the number of marked vertices, hence (\ref{eq:sum_marked}) holds for
$\CP_{u}'$. The inequality (\ref{eq:sum_marked}) immediately implies (\ref{eq:realization_mvertex_bound}).
The reduction process can introduce one auxiliary vertex, but since both $\CP_{u_1}$ and $\CP_{u_2}$
have a zero vertex, the bound (\ref{eq:sum_vert}) also holds for $|V(\CP_{u}')|$.
Now, it follows from (\ref{eq:sum_vert})
that every operation increases the number of marked vertices by at most $|M(\CP_u')|$,
which is bounded in view of (\ref{eq:realization_mvertex_bound}) by the number
$\sum_{x\in X} \sigma_x(u) \cdot |V(\eta(x))|  + 2\sigma_1(u) + \sigma_0(u)$.
Thus the inequality (\ref{eq:realization_vertex_bound}) holds.

Finally, we use Theorem \ref{th:normal_circuit} to compute the normal circuit
for $\CP_t'$. That increases the total number of vertices by up to a factor of $2$
and does not increase the number of marked vertices. Hence the required bounds for $|V(\CP_t)|$
and $|M(\CP_t)|$ follow.

The Algorithm \ref{al:term_realization} performs $|t|$
reductions. By Proposition \ref{pr:reduction_complexity} the complexity of reducing
a circuit $\CP$ requires $O(|V(\CP)|^3)$ steps.
Using the bound (\ref{eq:realization_vertex_bound})
we obtain the required bound on the time complexity of Algorithm \ref{al:term_realization}, which finishes the proof.

\end{proof}

\begin{corollary}\label{co:term_realization_Z}
Let $t(X) \in \CT_{L_0}$ and $\eta:X \to \MZ$ an assignment of variables.
There exists an algorithm which determines if $t(\eta(X))$ is defined in $\MZ_\CL$ (or $\tilde Z$)  or not;
and if defined it then produces the normal
circuit $\CP_t = \lambda(t(\eta(X))) \in\CC_{norm}$. The algorithm has time complexity
    $$O\rb{ |t|^4 \cdot \rb{ \sum_{x\in X} \sigma_x(t)(s_x+2)+2\sigma_1(t)+\sigma_0(t)}^3}$$
where $s_x = \lceil \log_2(|\eta(x)|+1)\rceil$.
\end{corollary}

\begin{proof}
The required algorithm first constructs the normal circuits $\CP_x$ representing
$\eta(x)$ for every $x\in X$ and then applies Algorithm \ref{al:term_realization}.
By Theorem \ref{th:normal_cp_number} the time complexity of computing
$\CP_x$ is $O(\log_2(s_x) \log_2\log_2(s_x))$ and $|V(\CP_x)| \le s_x+2$.
Application of Theorem \ref{th:term_realization_complexity} finishes the proof.
\end{proof}

\subsection{Quantifier-free formulas in exponential algebra}

In this section we study the quantifier-free theory of the $\CL_0$-structure
$$\tilde Z   = \langle \mathbb{Z} ; +,-, x\cdot 2^y, x\cdot2^{-y}, \leq, 0, 1\rangle$$
with all the constant from $\MZ$ in the language.
To this end we extend the language $\CL_0$ to
$\CL_0^{const}$ by adding  a constant symbol $n$ for every
$n \in \mathbb{Z}$. The structure $\tilde Z$ naturally
extends to a structure $\tilde{Z}_{const}$ in the language
$\CL_0^{const}$.  Since $\CL_0^{const}$ is an infinite language
the complexity of algorithmic problems in $\tilde{Z}_{const}$
depends on how we present the data, in this case, the constants
$n \in \mathbb{Z}$. We assume here that all the integers
$n \in \mathbb{Z}$ are given in their binary forms.
Of course, since every integer $n$ can be presented as a closed
term in the structure $\langle \mathbb{Z}, +, -, 0,1\rangle$,
every term $t$  in the structure $\tilde{Z}_{const}$ can be
presented by a term $t'$ in the structure $\tilde{Z}$, but in this case the
length of the term $t'$ can grow exponentially in the length of $t$.
In such  event one  would allow too much of leeway to himself
(when working on complexity problems)  by representing integers
in the unary form, and the results would be weaker.

\begin{theorem}\label{th:quantifier-free}
The quantifier-free theory of the structure  $\tilde{Z}_{const}$  is decidable in polynomial time.
\end{theorem}

\begin{proof}
A quantifier free formula in $\tilde{Z}_{const}$ is a formula of the type
$t_1(X) \diamondsuit t_2(X)$ where $t_1,t_2\in \CT_{L_0}$
and $\diamondsuit \in \{\le,=\}$. To determine if it holds in $\tilde{Z}_{const}$ it suffices to compute the normal power circuit representing the therm $t_1-t_2$ and  use
Proposition \ref{pr:complexity_compare_circ}
to compare the value $\CE(\CP)$ with $0$.
Both operations have polynomial time complexity in terms of the size of the formula, hence the result.
\end{proof}

\begin{corollary} \label{co:quantifier-free-no-constants}

The quantifier-free theory of the structure  $\tilde Z$ is decidable in polynomial time.
\end{corollary}

\begin{corollary}

The quantifier-free theory of the structure  $\tilde N = \langle \mathbb{N} ;+,-, x\cdot 2^y, x\cdot2^{-y}, \leq, 0, 1\rangle$  is decidable in polynomial time.
\end{corollary}

\section{Some inherent difficulties in  computing  with power circuits}
\label{se:difficulties}

In this section  we  demonstrate that  a product of power circuits may result in a power circuit  whose size  may grow exponentially in the size of the factors. We also show that solving some linear equations in power circuits may take super-exponential time.

\subsection{Division by 3}
\label{:se:modulo-arithmetic}

 For each natural $i$ consider a number
 $$N_i = 2^{2i} + 2^{2(i-1)} + \ldots +2^2 + 2^0 =  \frac{4^{i+1}-1}{3}.$$
 The binary sum above is compact, so by Lemma \ref{le:compact_sum} it is a shortest binary sum decomposition of $N_i$. Hence any other binary decomposition of $N_i$ contains at least $i+1$ terms. This implies that
any power circuit $\CP_i$ representing the number $N_i$ contains at
least $i+1$ vertices. Now, pick
$$i = tower_2(j) =
2^{\left.2^{\ldots^2}\right\}j~ \mbox{times}}.$$
Then  $3N_i =
4^{i+1}-1$ and there exists a circuit, say $\CP_j$, on
$j+1$ vertices representing $4^{i+1}-1$. This follows that the linear equation $3x = \CP_j$ has a solution  $\CP_i$ in the power circuit arithmetic, but any power circuit that gives a solution of this equation has at least $i = tower_2(j)$ vertices.   This proves the following proposition.

\begin{proposition} \label{pr:division-by-3}
The worst case complexity of solving a linear  equation $3x = \CP$ in power circuits ($\CP$ is a constant and $x$ is a variable over  the set of power circuits)  is super-exponential.
\end{proposition}

We conclude this section with an observation that  prime factorization of numbers given by power circuits can be super-exponential.

\subsection{Power circuits and multiplication}
\label{se:with-multiplication}

In this section  we  demonstrate some inherent difficulties when dealing with  products of power circuits (the size of the resulting circuit grows exponentially).

For $n\in\MN$ define a power circuit $\CP_n  =(\CP_n,\mu,M,\nu)$, where
\begin{itemize}
    \item
$\CP_n = (V_n,E_n)$ and $V_n = \{0,\ldots,n\}$
and $E_n = \{(i,i-1) \mid i=1,\ldots,n\} \subset V_n\times V_n$;
    \item
$\mu \equiv 1$;
    \item
$M = \{1,n\}$;
    \item
$\nu(1) = \nu(n) = 1$.
\end{itemize}
Clearly, $\CE(\CP_n) = \tower_2(n-1)+1$ and $|\CP_n| = n+1$.
The product $\CP_4 \cdot \ldots \cdot \CP_{n}$  represents the number
    $$\prod_{i=4}^n (\tower_2(i-1)+1) = \sum_{\sigma\in \{0,1\}^{n-3}} \rb{\prod_{4\le i\le n,~\sigma_{i-3}=1} \tower_2(i-1)}$$
    $$ = \sum_{\sigma\in \{0,1\}^{n-3}} \rb{\prod_{4\le i\le n,~\sigma_{i-3}=1} 2^{\tower_2(i-2)}} = \sum_{\sigma\in \{0,1\}^{n-3}} 2^{s_\sigma}$$
where
    $$s_\sigma = \sum_{4\le i\le n,~\sigma_{i-3}=1} \tower_2(i-2).$$
The binary sum $\sum_{\sigma\in \{0,1\}^{n-3}} 2^{s_\sigma}$ is compact and hence
by Lemma \ref{le:compact_sum} involves the least number $2^{n-3}$ of terms.
Therefore the product $\CP_4 \cdot \ldots \cdot \CP_{n}$
can not be represented by a power circuit of size less than $2^{n-3}$.

\section{Open Problems}
\label{se:open_problems}

In this section we state some interesting algorithmic problems for exponential algebras.

\begin{problem}
Can one develop a robust theory of power circuits when $\mathbb{Z}$ is replaced by  $\mathbb{Q}$? or $\mathbb{R}$?
\end{problem}

Here the main concern is the reduction algorithm.

\begin{problem}
\begin{enumerate}
 \item [1)] Is the quantifier-free theory of the standard high-school arithmetic  $\mathbb{N}_{HS}$  polynomial time decidable (with all constants from $\mathbb{N}$ in the language)?
\item [2)] Is the equational theory of $\mathbb{N}_{HS}$  polynomial time decidable?
\end{enumerate}
\end{problem}

The example in Section \ref{se:with-multiplication} demonstrates that power circuits in the structure $\mathbb{N}_{HS}$ do not allow fast manipulations that involve arbitrary multiplications. Nevertheless, it might be that there are some other means to approach the problem.

\begin{problem}
Is the existential theory of $\tilde N  = \langle \mathbb{N}_{>0}; +,  x\cdot 2^y,  \leq,  1\rangle$ decidable?
Is the Diophantine problem decidable?
\end{problem}

\begin{problem}
What is the time complexity of the  the problem of finding a minimal (in size) constant power circuit  representing a given natural number?
\end{problem}

\begin{problem}
Is $\tilde N$ automatic?
\end{problem}


\providecommand{\bysame}{\leavevmode\hbox to3em{\hrulefill}\thinspace}
\providecommand{\MR}{\relax\ifhmode\unskip\space\fi MR }
\providecommand{\MRhref}[2]{%
  \href{http://www.ams.org/mathscinet-getitem?mr=#1}{#2}
}
\providecommand{\href}[2]{#2}

\end{document}